\documentclass{amsart}

\usepackage{tikz,subfigure}

\newtheorem{theorem}{Theorem}

\newtheorem{corollary}[theorem]{Corollary}
\newtheorem{lemma}[theorem]{Lemma}
\newtheorem{proposition}[theorem]{Proposition}
\newtheorem{remark}[theorem]{Remark}
\newtheorem{definition}[theorem]{Definition}

\newcommand{\eps}{\varepsilon}
\newcommand{\la}{\lambda}
\newcommand{\laa}{\lambda_1^{A}}
\newcommand{\smax}{\sigma_{\text{max}}}

\newcommand\Prob{{\mathcal P}}
\newcommand\cOmega{\overline{\Omega}}
\newcommand\sfera{{{\mathbb S}^1}}
\newcommand\cil{\cOmega\times\sfera}
\newcommand\Csym{C_{\text{sym}}(\overline{\Omega}\times {\mathbb S}^1)}
\newcommand{\R}{\mathbb{R}} 
\newcommand{\Si}{\Sigma}
\newcommand{\Om}{\Omega}
\newcommand{\Oms}{\Omega\setminus\Sigma}
\newcommand{\Ots}{\overline{\Omega}\times\mathbb{S}^1}
\newcommand{\Alo}{\mathcal{A}_L(\Omega)}
\newcommand\chb{{\mathcal Q}_s}
\newcommand{\prob}{\mathcal{V}_1(\overline{\Omega})}
\newcommand{\weak}{\rightharpoonup}
\newcommand{\haus}{\mathcal{H}^1}

\newcommand{\rectangle}{

\draw[-, line width=0.6pt] (0,0)--(0,10); %RETTANGOLI
\draw[-, line width=0.6pt] (0,10)--(10,10);
\draw[-, line width=0.6pt] (10,10)--(10,0);
\draw[-, line width=0.6pt] (10,0)--(0,0);
\draw[-,very thin] (0,2)--(10,2);
\draw[-,very thin] (0,7)--(10,7);

\draw[-,very thin] (0,9.3)--(0.5,10);
\draw[-,very thin] (1,7)--(3,10);
\draw[-,very thin] (3.5,7)--(5.5,10);
\draw[-,very thin] (6,7)--(8,10);
\draw[-,very thin] (8.5,7)--(10,9.2);

\draw[-,very thin] (0,3)--(2.375,2);
\draw[-,very thin] (0,4.5)--(5.94,2);
\draw[-,very thin] (0,6)--(9.5,2);
\draw[-,very thin] (1.187,7)--(10,3.29);
\draw[-,very thin] (4.75,7)--(10,4.79);
\draw[-,very thin] (8.31,7)--(10,6.29);

\draw[-,very thin] (0,1.5)--(1.5,0);
\draw[-,very thin] (1.5,2)--(3.5,0);
\draw[-,very thin] (3.5,2)--(5.5,0);
\draw[-,very thin] (5.5,2)--(7.5,0);
\draw[-,very thin] (7.5,2)--(9.5,0);
\draw[-,very thin] (9.5,2)--(10,1.5);
}

\DeclareMathOperator*{\esssup}{ess\,sup}
\DeclareMathOperator*{\essinf}{ess\,inf}

\newdimen\mex
\makeatletter
\def\niv{\mathrel{\hbox{\hglue -0.4\mex
\vrule \@height 1.4\mex \@width .14\mex
\vrule \@height .14\mex \@width .75\mex
\hglue -0.2\mex}}}
\makeatother

\title[Where best to place a Dirichlet condition]{Where best to place a Dirichlet condition
in
an anisotropic membrane?}
\author{Paolo Tilli, Davide Zucco}

\address{Paolo Tilli, Dipartimento di Scienze Matematiche,  Politecnico di Torino, Italy}
\email{paolo.tilli@polito.it}

\address{Davide Zucco, Scuola Internazionale Superiore di Studi Avanzati, Trieste, Italy}
\email{davide.zucco@sissa.it}

\begin{document}

\begin{abstract}
We study a shape optimization problem for the first eigenvalue of
an elliptic operator in divergence form, with non constant coefficients,
over a fixed domain $\Om$.
Dirichlet conditions are imposed along $\partial \Om$ and, in addition,
along a   set $\Si$ of prescribed length ($1$-dimensional Hausdorff measure).
We look for the best shape and position for the
supplementary Dirichlet region $\Si$ in order to maximize the first eigenvalue.
The limit distribution of the optimal sets, as their
prescribed  length tends to infinity, is characterized  via $\Gamma$-convergence
of suitable functionals defined over varifolds: the use of varifolds, as opposed
to probability measures, allows one to keep track of the local
orientation of the optimal sets (which comply with the anisotropy of the problem),
and not just of their limit distribution.
\end{abstract}

\maketitle

\section{Introduction}
We investigate an optimization problem for the first Dirichlet
eigenvalue of an elliptic operator in divergence form, where the
unknown is a compact set $\Sigma$ of prescribed length $L$, along
which a homogeneous Dirichlet condition is imposed. Our interest
is focused on the behavior of the optimal sets, when the
prescribed length $L\to\infty$.

The case of the Laplacian was studied in \cite{tilzuc}.
Here, on the contrary, the
elliptic operator is anisotropic, and this requires
the introduction of new ideas and techniques: among them,
the main novelty is the use of \emph{varifolds}
(instead of probability measures), as a tool to properly handle the anisotropy of the problem.

Throughout we shall assume, without further reference, that we are given:
\begin{itemize}
\item [(i)] a bounded domain
$\Omega\subset\R^2$ with a Lipschitz boundary $\partial\Omega$ (we do not assume $\Om$ to be simply connected);
\item [(ii)] a $2\times 2$ symmetric, positive definite, matrix-valued
function $A(x)=[a_{ij}(x)]$, defined for $x\in\overline{\Omega}$
and continuous there.
\end{itemize}
Observe that $A(x)$ is \emph{uniformly elliptic}, i.e. for some constant $C>0$
\begin{equation}
\label{unifell}
C^{-1} |y|^2\leq \langle A(x) y,y\rangle \leq C|y|^2\quad
\forall y\in\R^2,\quad\forall x\in\overline{\Omega}.
\end{equation}
For every $L>0$, the class of admissible sets
(the ``supplementary Dirichlet regions'') is defined as follows:
\begin{equation}\label{class}
\Alo:=\left\{\Sigma\subset\overline{\Omega}\,\,:\,\,\text{$\Sigma$ is a continuum and $\haus(\Sigma)\leq L$}\right\}
\end{equation}
(throughout, ``continuum'' stands for ``non-empty, connected, compact set'' while
$\haus$ denotes the one-dimensional Hausdorff measure).
We are interested in the maximization problem
\begin{equation}\label{prob}
 \max \left\{\laa(\Om\setminus\Si): \; \Sigma\in \Alo\right\},
\end{equation}
where the first eigenvalue $\laa(\Oms)$ is given by
\begin{equation}\label{eigen}
\laa(\Oms)= \min_{\genfrac{}{}{0pt}{}{u\in H^{1}_0(\Oms)}  {u\not\equiv 0}} \frac{\int_{\Om} \langle A(x)\nabla u(x),\nabla u(x)\rangle dx}{\int_{\Om}u(x)^2 dx}
\end{equation}
(since any admissible $\Sigma$ has zero Lebesgue measure, integration
over $\Omega$ is equivalent to integration over $\Omega\setminus\Sigma$).
For fixed $\Sigma\in\Alo$,
the coefficient matrix
$A(x)$ gives rise, in the open set $\Omega\setminus \Sigma$,
to the elliptic operator
%\begin{equation}\label{operator}
%-\sum_{i,j=1}^2\frac{\partial}{\partial x_i}\big(a_{ij}(x)u_{x_j}(x)\big)
%\end{equation}
$-\mathop{\rm div}A(x)\nabla u$:
the number $\lambda_1^A(\Omega\setminus\Sigma)$ is then the first eigenvalue relative
to this operator, with a homogeneous Dirichlet condition along $\Sigma\cup\partial\Omega$.
Every set $\Sigma\in\Alo$ is rectifiable and (except for the uninteresting case where
it is a singleton) it
has a \emph{positive capacity}: therefore, the condition $u\in H^1_0(\Omega\setminus\Sigma)$
is in fact equivalent to $u\in H^1_0(\Omega)$ with the \emph{additional} condition that
$u=0$ along $\Sigma$, a.e. with respect to the Hausdorff measure $\haus$.
Thus, if $\Omega$ is a membrane fixed along its boundary $\partial\Omega$,
$\Sigma$ can be interpreted as a stiffening rib,
that will
increase the fundamental frequency of the membrane.
Observe that, by \emph{monotonicity} with respect to the domain,
%$\lambda_1^A(\Omega\setminus\Sigma)\leq \lambda_1^A(\Omega\setminus\Sigma')$ if $\Sigma\subseteq \Sigma'$,
%so that
the larger $\Sigma$, the higher the first eigenvalue.

In \eqref{class},
the \emph{length constraint} $\haus(\Sigma)\leq L$, combined with the connectedness
assumption,
prevent $\Sigma$ to become too large and spread out over $\Omega$:
these constraints can also be seen as a bound on the total resources available, in order to raise the membrane
frequency (for related problems in the same setting, see e.g. \cite{butsan,mostil}).
Without any length constraint, the best choice would be $\Sigma=\overline{\Omega}$, and
the problem would be trivial: in \eqref{prob}, however, the optimal sets $\Sigma_L$ still tend
to \emph{saturate} $\Omega$ as much as possible, compatibly with the constraint that $\Sigma_L\in\Alo$.
In this paper we shall investigate \emph{how} this saturation occurs
as $L\to\infty$, answering questions such as:
with what limit density (length per unit area) will $\Sigma_L$ saturate $\Omega$? With what local orientation?

Due to Go{\l}ab and Blaschke theorems
(see \cite{ambtil}),
$\Alo$ is compact
with respect to the Hausdorff convergence, while the first
eigenvalue $\laa$ is continuous (see \cite{sve,hen}). Therefore,
the existence of an optimal set $\Sigma_L$
for problem \eqref{prob} is immediate:

\begin{theorem}
For every $L>0$ there exists a maximizer in \eqref{prob}. Moreover,
every maximizer $\Sigma_L$ satisfies $\haus(\Sigma_L)=L$.
\end{theorem}

The last part of the claim is typical of constrained problems that obey monotonicity with
respect to domain inclusion: if $\haus(\Si)< L$, then $\laa(\Oms)$ could be increased,
e.g. by attaching
to $\Si$  some short segment (see  \cite{buoust, butsan,tilzuc}). Observe, en passant, that the corresponding
\emph{minimization} problem is trivial: the optimal sets are all singletons $\Sigma=\{x_0\}$ and
all sets $\Sigma\subseteq\partial\Omega$ (in both cases, one has $\lambda_1^A(\Omega\setminus\Sigma)=
\lambda_1^A(\Omega)$ which is clearly optimal).

\medskip

Our goal is to study the limit behavior of the optimal sets $\Sigma_L$ as $L\to\infty$, via
$\Gamma$-convergence. If $A(x)=\sigma(x) I$ is a multiple of the identity matrix, i.e. if the
membrane $\Omega$ is \emph{isotropic} (though not necessarily uniform),
this has been investigated in a previous paper \cite{tilzuc}. Passing from $A(x)=\sigma(x) I$
to an anisotropic matrix $A(x)$, however, is not a merely technical extension: indeed,
several new ideas are required, and even the final expression for the $\Gamma$-limit functional
is not easy to figure out.

In \cite{tilzuc},
along lines typical of this kind of problems (see \cite{bouchitte,bouchitte2,butsan,butsanstep,mostil}),
the main idea
was to associate, with every admissible set $\Sigma_L\in\Alo$,
the \emph{probability measure} $\mu_\Sigma$ defined as
\begin{equation}\label{mul}
 \mu_\Si:=\frac{\haus\niv\Si}{\haus(\Sigma)},
 \end{equation}
where $\haus\niv\Si$ is the one-dimensional Hausdorff measure restricted to $\Sigma$.
If $\Sigma_L$ is a solution to \eqref{prob} and $L\to\infty$, knowledge of the weak-* limits $\mu$ of the corresponding
measures $\mu_{\Sigma_L}$ in the space of probability measures $\Prob(\overline{\Omega})$, provides
information on \emph{how} the optimal sets $\Sigma_L$ tend to saturate $\Omega$ in the limit.
In concrete terms, one can define the functionals $G_L:\mathcal P(\overline{\Omega})\to [0,+\infty]$
\begin{equation}\label{fl1}
G_L(\mu)=
\begin{cases}
\dfrac{L^2}{\laa(\Oms)} &\text{if  $\mu=\mu_\Si$ for some $\Si\in\Alo$,}\\[2mm]
+\infty  &\text{otherwise,}
\end{cases}
\end{equation}
and investigate their $\Gamma$-limit in $\Prob(\cOmega)$ as $L\to\infty$.
Indeed, for fixed $L$,
the \emph{maximization} problem \eqref{prob} is equivalent to the
\emph{minimization} of $G_L$, while the factor $L^2$ is a renormalization
that allows for a nontrivial $\Gamma$-limit. If the coefficient matrix $A(x)=\sigma(x)I$
is isotropic and $\sigma(x)$ is continuous, it was proved in \cite{tilzuc} that the functionals $G_L$ $\Gamma$-converge
(with respect to the weak-* topology on $\mathcal P(\overline{\Omega})$) to the
functional
\begin{equation*}
%\label{vecchioGL}
G_\infty(\mu):=
\frac{1}{\pi^2}\esssup_{x\in\Omega}\frac{1}{f_\mu(x)^2\sigma(x)},
\end{equation*}
where $f_\mu(x)$ is the density of $\mu$  w.r.to the Lebesgue measure (note
the interaction between $f_\mu(x)$ and $\sigma(x)$,
the Young modulus of the membrane). This functional is minimized by a \emph{unique} probability
measure, an $L^1$ function proportional to $\sigma(x)^{-1/2}$: as a consequence (\cite{tilzuc}),
one infers that for large $L$ the optimal sets $\Sigma_L$ have, at small scales, a
comb-shaped structure, essentially made up of parallel lines whose spacing is locally proportional to
$\sigma(x)^{1/2}$. The local orientation of the comb teeth is,
on the other hand, \emph{irrelevant},
and rotating an optimal pattern preserves optimality: this clearly reflects the isotropy of
the coefficient matrix $A(x)=\sigma(x)I$.

\medskip

When the matrix $A(x)$ is \emph{anisotropic}, the previous analysis
is no longer valid. One still expects optimal sets to look like comb-shaped patterns
at small scales, but their local orientation will now depend on $A(x)$: heuristically,
since in the smooth case the gradient $\nabla u(x)$ of the first eigenfunction $u\in H^1_0(\Omega\setminus\Sigma)$
is roughly directed as the normal unit vector $\xi(x)$ at $x\in\Sigma$, one expects $\Sigma$ to be locally
directed in such a way that $\xi(x)$ is parallel to an eigenvector of $A(x)$, namely the one relative
to its \emph{largest} eigenvalue, so that the quadratic form $\langle A(x)\nabla u,\nabla u\rangle$
is (locally) maximized in \eqref{eigen}.
On the other hand,
the probability measures $\mu$ in \eqref{mul} and their weak-* limits keep \emph{no track}
of
the orientation of the optimal $\Sigma_L$: they just keep memory of the density (length per unit area) of
$\Sigma_L$ inside $\Omega$. In other words, to fully understand the anisotropic case
we have to associate, with every admissible $\Sigma\in\Alo$, a more complex object than a
measure $\mu\in \mathcal P(\overline{\Omega})$, possibly one that carries information also on the \emph{unit normal}
to $\Sigma$, and study $\Gamma$-convergence in that setting.

Roughly speaking, the main idea is to associate,
with every
$\Si\in\Alo$, a probability measure $\theta_\Si$ in the
\emph{product space} $\overline{\Omega}\times
 \mathbb{S}^1$,   defined via integration by
\begin{equation}
\label{defT}
\int_{\Ots} \varphi(x,y)d\theta_\Sigma(x,y):=\frac{1}{\haus(\Si)}
\int_{\Si}\varphi\big(x,\xi_{\Si}(x)\big)\,d\haus(x),\quad
\varphi\in C_{\text{sym}}(\overline{\Omega}\times
 \mathbb{S}^1),
\end{equation}
where $\xi_\Sigma(x)$ is the unit normal vector to $\Sigma$ at $x\in\Sigma$, while
\[
C_{\text{sym}}(\overline{\Omega}\times {\mathbb S}^1):=
\left\{
\varphi\in
C(\overline{\Omega}\times {\mathbb S}^1)\,\,
\vert\,\, \varphi(x,y)=\varphi(x,-y)\quad\forall x\in\overline{\Omega},
\forall y\in {\mathbb S}^1\right\}
\]
is the space of continuous functions with \emph{antipodal} symmetry.
Since every set
$\Si\in\Alo$ is 1-rectifiable, a unit normal $\pm\xi_\Sigma(x)$ is well defined, up to the orientation,
at $\haus$-a.e. $x\in\Sigma$, and $\xi_\Sigma(x)$ in \eqref{defT} denotes any \emph{measurable selection}
of this normal field: due to the antipodal symmetry of $\varphi$,
the second integral in \eqref{defT} does \emph{not} depend on the particular selection.

In fact, \eqref{defT} defines $\theta_\Sigma$ as an element of the dual space of
$C_{\text{sym}}(\overline{\Omega}\times {\mathbb S}^1)$,
that is, as a (probability) measure over
$\overline{\Omega}\times \mathbb P^1$, where $\mathbb P^1$ is the \emph{projective space}.
% On the other hand, due to this restriction on $\varphi$, \eqref{defT}
%is not enough to define $\Theta_\Sigma$ as a \emph{measure} over $\overline{\Omega}\times {\mathbb S}^1$
%(for this, one should define $\langle \Theta_\Sigma,\varphi\rangle$ for \emph{every} continuous
%$\varphi$): in fact, \eqref{defT} completely defines $\Theta_\Sigma$ as a measure over
%$\overline{\Omega}\times \mathbb P^1(\R)$, where $\mathbb P^1(\R)$ is the \emph{projective space}.
In other words, \eqref{defT} defines  $\theta_\Sigma$ as the
$1$-dimensional \emph{varifold} induced by $\Sigma$ (see \cite{alm,sim}),
with the proviso that we work with the \emph{normal} instead of the more usual
tangent space.

We denote by $\prob$ the space of $1$-dimensional varifolds with unit mass, that is,
probability measures over $\overline{\Omega}\times \mathbb P^1$, endowed with the usual
weak-* topology.
Throughout, however, we shall always consider a  varifold $\theta\in\prob$
as
an \emph{equivalence class} of probability measures on
$\overline{\Omega}\times  \mathbb{S}^1$, two measures being equivalent $\iff$ they
induce the same linear functional on  $\Csym$: thus,
by choosing a representative in its equivalence class,
we can still treat $\theta\in\prob$ as a
probability measure over $\overline{\Omega}\times {\mathbb S}^1$. With this agreement
(useful to avoid the technical language of varifolds), the weak-* convergence
of a sequence $\{\theta_L\}$ in $\prob$ to a varifold $\theta\in\prob$
takes the concrete form
\begin{equation}
\label{weakconvergenceT}
%\text{$\Theta_n\rightharpoonup \Theta$ in $\prob$}
%\quad\iff\quad
\lim_{L\to\infty}
\int_{\Ots} \varphi(x,y)
 \,d\theta_L
=
\int_{\Ots} \varphi(x,y)
\,d\theta\quad
\forall
\varphi\in\Csym.
\end{equation}

Now we are ready to state our main result.
By analogy with \eqref{fl1}, for $L>0$ we define the functional
$F_L:\prob\to [0,\infty]$ as
\begin{equation}\label{fl}
F_L(\theta)=
\begin{cases}
\dfrac{L^2}{\laa(\Oms)} &\text{if  $\theta=\theta_\Si$ for some $\Si\in\Alo$,}\\[2mm]
+\infty  &\text{otherwise,}
\end{cases}
\end{equation}
where $\theta_\Sigma$ is as in \eqref{defT}. The definition of the $\Gamma$-limit
functional $F_\infty$, which of course should depend on the coefficient matrix $A(x)$, is more involved.

\begin{definition}\label{defFi}
Given $\theta\in\prob$, by choosing a representative in its equivalence class we can
regard $\theta$ as an element of $\Prob(\cil)$,
and we can consider its \emph{first marginal} $\mu$, i.e. the probability measure
over $\overline\Omega$ defined by
\[
\mu(E):=\theta(E\times {\mathbb S}^1),\quad
\text{for every Borel set $E\subseteq\overline\Omega.$}
\]
Then, by well known results (see \cite{amfupa}), one can \emph{disintegrate}
$\theta$ as $\mu\otimes\nu_x$, where $\{\nu_x\}$ is a
$\mu$-measurable
family of probability
measures over ${\mathbb S}^1$ defined for $\mu$-a.e. $x\in\overline\Omega$. This means
that
\begin{equation*}
%\label{defslice}
\int_{\Ots} \varphi(x,y)d\theta=
\int_{\overline\Omega}
\left(
\int_\sfera \varphi(x,y)\,d\nu_x(y)\right)
\,d\mu(x)
,\quad
\forall\varphi\in C(\overline{\Omega}\times
 \mathbb{S}^1).
\end{equation*}
The measure $\mu$ depends only on $\theta$ as an element of
$\prob$ (not on the choice of its representative in $\Prob(\cil)$), while
the measures $\nu_x$ may well depend on the particular representative.
However, the integrals
\begin{equation*}
%\label{integrals}
x\mapsto \int_\sfera \sqrt{\langle A(x) y,y\rangle} \,d\nu_x(y)\quad
\text{(defined for $\mu$-a.e. $x\in\cOmega$)}
\end{equation*}
are independent of the particular representative of $\theta$, since the function
$(x,y)\mapsto \langle A(x) y,y\rangle^{1/2}$ belongs to $\Csym$. As a consequence,
if $f_\mu(x)$ denotes the density of $\mu$ (w.r.to the Lebesgue measure on $\cOmega$),
the functional
\begin{equation}\label{gammalimit}
F_\infty(\theta):=\frac{1}{\pi^2}\esssup_{x\in\Omega}\frac{1}
{\big(f_\mu(x)\int_{\mathbb{S}^1} \sqrt{\langle A(x)y,y\rangle}\, d\nu_x(y) \big)^2},\qquad
\theta\in\prob
\end{equation}
is well defined over $\prob$ (with values in $[0,+\infty]$), since it only depends on
$\theta$ as an element of $\prob$.
\end{definition}

The main result of the paper is then the following.

\begin{theorem}\label{aniso}
As $L\to\infty$, the functionals $F_L$ defined in \eqref{fl}
$\Gamma$-converge, with respect to the weak-*
topology on $\prob$, to the functional $F_\infty$ defined in \eqref{gammalimit}.
\end{theorem}

This has several consequences, as soon as we single out all those varifolds $\theta\in \prob$
that minimize the functional $F_\infty$.
For $x\in\cOmega$, we let $\sigma_{\max}(x)$ denote the \emph{largest eigenvalue}
of $A(x)$, while $\sigma_{\min}(x)$ denotes its smallest eigenvalue.

\begin{theorem}\label{teomin}
A varifold  $\theta_\infty\in\prob$ is a minimizer of $F_\infty$ if and only if
it can be disintegrated as
$\theta_\infty=f_\infty(x)\,dx\otimes
\nu^\infty_x$, where $f_\infty$ is the $L^1$ function defined as
\begin{equation}
\label{defbestf}
f_\infty(x)=\frac{1/\sqrt{\sigma_{\max}(x)}}{\int_\Om 1/\sqrt{\sigma_{\max}(x)}dx}
\end{equation}
while, for a.e. $x\in\Omega$,  $\nu^\infty_x\in\Prob(\sfera)$ is any probability measure
supported on a set made up of (normalized) eigenvectors of $A(x)$, relative
to $\sigma_{\max}(x)$.

The minimizer is unique if and only if $A(x)$ is \emph{purely anisotropic}, that is, if
and only if $\sigma_{\min}(x)<\sigma_{\max}(x)$ occurs at a.e. $x\in\Omega$. In any case,
the first marginal $\mu$ of a minimizer $\theta_\infty$ is necessarily the function in \eqref{defbestf}.
\end{theorem}

As $\prob$ is compact in the weak-* topology, from standard $\Gamma$-convergence theory (see~\cite{dalmaso})
we have
\begin{corollary}\label{cordist}
For $L>0$, let $\Si_L$ be a maximizer of problem \eqref{prob} and let $\theta_{\Sigma_L}$ be the associated
varifolds, according to \eqref{defT}.
Then, as $L\to\infty$, $\theta_{\Si_L}\weak\theta_\infty$
(up to subsequences) in the weak-* topology of $\prob$,
where $\theta_\infty$ is a minimizer of $F_\infty$.
\end{corollary}

In spite of its abstract formulation, this corollary has interesting consequences in
the applications. In particular, it enables us to determine the asymptotic density
(length per unit area) of the optimal sets $\Sigma_L$, as well as their \emph{local}
direction (distribution of the unit normal on the projective space),  when $L\to\infty$:
\begin{theorem}\label{teodist}
For $L>0$, let $\Si_L$ be a maximizer of problem \eqref{prob}.
For every square $Q\subset\Omega$,
\begin{equation}\label{asopt}
\lim_{L\to\infty}\frac{\haus(\Sigma_L\cap Q)}{\haus(\Si_L)}=\int_Q f_\infty(x)\,dx\qquad
\text{(optimal density of length)}
\end{equation}
where $f_\infty$ is as in \eqref{defbestf}. Moreover,
if $Q$ is contained in the anisotropy region where $\sigma_{\min}(x)<\sigma_{\max}(x)$,
for every $\psi\in C(\sfera)$ such that $\psi(-y)=\psi(y)$ one has
\begin{equation}
\label{normopt}
\lim_{L\to\infty}\frac 1  {\haus(Q\cap\Sigma_L)}
\int_{Q\cap\Sigma_L} \psi\bigl(\xi_L(x)\bigr) \,d\haus(x)
=
\dfrac{\int_Q f_\infty(x)\psi\bigl(\xi(x)\bigr)\,dx}
{\int_Q f_\infty(x)\,dx},
\end{equation}
where $\xi_L(x)$ is the unit normal to $\Sigma_L$ at $x\in\Sigma_L$, while $\xi(x)$ is
the (unique up
to the orientation) eigenvector of $A(x)$  relative to $\sigma_{\max}(x)$.
Finally,
\begin{equation}\label{asymptotically}
 \lim_{L\to \infty} \frac{L^2}{\laa(\Oms_L)}=\frac{\left(\int_{\Om}1/\sqrt{\sigma_{\max}(x)}dx\right)^2}{\pi^2}.
\end{equation}
\end{theorem}

The meaning of \eqref{asopt} is that, if $L$ is large, in order to solve \eqref{prob} one
should spread $\Sigma$ over $\Omega$ with a density (length per unit area) roughly proportional
to $f_\infty(x)$. Moreover, by \eqref{normopt}, the optimal set $\Sigma_L$ should be
directed in such a way that its unit normal $\xi_L(x)$ is essentially parallel to the
eigenvector $\xi(x)$ relative to the largest eigenvalue of $A(x)$ (this is uniquely determined only
in the region where $A(x)$ is anisotropic).

The initial assumption on the continuity of $A(x)$ plays an
important role in our proofs, since it enables us to localize our
estimates,  as if $A(x)$ were  constant at small scales (the
analog problem when $A(x)$ is merely measurable and satisfies
\eqref{unifell} is entirely open). On the other hand, all our
results are easily extended to cover the case where $A(x)$ is
\emph{piecewise continuous}, i.e. continuous over disjoint open
sets $\Omega_i$ with smooth boundaries, such that $\cOmega=\bigcup
\cOmega_i$. This extension would be a merely technical one and is
not pursued in detail.

Finally, we point out that regularity of the optimal sets $\Sigma_L$ is
an open problem. In particular, it is not known whether $\Sigma_L$ satisfies
no-loop or blow-up properties, similar to those of the minimizers
of the irrigation problem (see \cite{butste,santil}).

\bigskip
\textbf{Notation}
For a Borel set $E\subset \R^2$ we denote by $|E|$  the two-dimensional Lebesgue measure
of $E$ and by $\haus(E)$ its one-dimensional Hausdorff measure.
If $X$ is a compact space, $\Prob(X)$ denotes the space of all (Borel) probability measures over $X$.
If $\mu$ is a measure, ``$\mu$-a.e.'' is an abbreviation of ``$\mu$-almost everywhere''
(if $\mu$ is omitted, Lebesgue measure is understood), while
``w.r.to $\mu$'' stands for ``with respect to $\mu$''.

\section{Some auxiliary estimates}

In this section we prove some preliminary results
on the first Dirichlet eigenvalue,
when the coefficient matrix $A(x)$ is constant
and  the Dirichlet condition is prescribed along a generic compact set $D$.

More precisely, let $D\subset\cOmega$ be a compact set with finitely many
connected components, such that
$0<\haus(D)<\infty$: such a set is $1$-rectifiable (see \cite{ambtil}) and has positive capacity, therefore
given $u\in H^1(\Omega)$ the Dirichlet condition
\[
u(x)=0\quad\text{for $\haus$-a.e. $x\in D$}
\]
is meaningful and defines a nontrivial closed subspace of $H^1(\Omega)$
(it can be interpreted as a trace condition, since by rectifiability every connected component
of $D$ can be covered by a Lipschitz curve). Then, one can define the first Dirichlet eigenvalue
of the domain $\Omega$
(relative to the coefficient matrix $A(x)$) with Dirichlet condition along $D$, as follows:
\begin{equation}
\label{lDD}
\lambda_1^A(\Omega;D)
:= \min_{\genfrac{}{}{0pt}{}{u\in H^{1}(\Omega)\setminus\{0\}} {\text{$u\equiv 0$ on $D$}}}
\frac{\int_{\Om} \langle A(x)\nabla u(x),\nabla u(x)\rangle dx}{\int_{\Om}u(x)^2 dx}.
\end{equation}
Observe that this is more general than \eqref{eigen}, where the minimization takes place
in a $H^1_0$ environment: given $\Sigma$, the eigenvalue $\lambda_1^A(\Omega\setminus\Sigma)$
in \eqref{eigen} is a particular case of $\lambda_1^A(\Omega;D)$, namely when $D=\Sigma\cup\partial\Omega$
(as $\Sigma$ is connected while $\partial\Omega$, being Lipschitz, has finitely many connected
components, the same is true of $D$).

When $A(x)\equiv M$ is constant,
equal to a $2\times 2$ positive definite matrix $M$
independent of $x$, it is possible to give an upper bound to $\lambda_1^M(\Omega;D)$ in
terms of some geometric quantities, among which an  important role is played by the
``Riemannian'' length
 \begin{equation}\label{hausm}
\haus_M(D):=\int_{D} \sqrt{\langle M\xi(x),\xi(x)\rangle}\,  d\haus(x),
\end{equation}
where $\xi(x)$ is (a measurable selection of) the unit \emph{normal} to $D$ at the point $x\in D$
(by rectifiability, this normal is well defined, up to the orientation, at $\haus$-a.e. $x\in D$).
\begin{theorem}\label{upper}
Let $M$ be a positive definite $2\times 2$ matrix, and let
$D\subset\cOmega$ be a compact connected set with $\kappa$ connected components, such that $0<\haus(D)<\infty$. Then
%Given $L>0$ and $\Si\in\Alo$, it holds
\begin{equation}\label{estb}
\lambda_1^M(\Omega;D)\leq\frac{{\pi^2}}{4t^2}\left(1+\frac{\kappa\pi t\det M^{1/2}}{\haus_M(D)}\right),
\end{equation}
where the number $t$ is defined as
\begin{equation}\label{deft}
 t=\dfrac{|\Om|}{\haus_M(D) +\sqrt{\haus_M(D)^2+\kappa\pi |\Om|\det M^{1/2}}}.
\end{equation}
\end{theorem}
\begin{proof}
Set  $N=M^{-1/2}$. Recalling \eqref{lDD}, by the change of variable $z=Nx$ in the two integrals
we have
\begin{equation*}%\label{other}
 \lambda_1^M(\Omega;D)=
\min_{\genfrac{}{}{0pt}{}{u\in H^{1}(\Omega)}{\text{$u\equiv 0$ on $D$}}}
\frac{\int_{\Om} \langle M\nabla u(x),\nabla u(x)\rangle dx}{\int_{\Om}u(x)^2 dx}
=
 \min_{\genfrac{}{}{0pt}{}{v\in H^{1}({N\Omega})}
{\text{$v\equiv 0$ on $ND$}}
} \frac{\int_{N\Om}  |\nabla v(z)|^2 dz}{\int_{N\Om} v(z)^2dz},
 \end{equation*}
so that $\lambda_1^M(\Omega;D)$ is
the first eigenvalue of the Laplacian
on the deformed domain $N\Omega$ (the image of $\Omega$ through the linear map $N$),
with Dirichlet conditions along $ND$ (observe that also $ND$ has $\kappa$ connected components).
Therefore, applying Theorem 2.4 of \cite{tilzuc} combined with Remark~2.5 therein,
we obtain
\begin{equation}\label{estb2}
\lambda_1^M(\Omega;D)
\leq\frac{{\pi^2}}{4 T^2}\left(1+\frac{\kappa\pi T}{\haus(ND)}\right),
\end{equation}
where
\begin{equation}\label{tmsecond}
T:=\dfrac{|N\Om|}{\haus(ND) +\sqrt{\haus(ND)^2+\kappa\pi|N\Om|}}.
\end{equation}
Therefore, to prove \eqref{estb}, it suffices to relate $T$ and $t$.

Using the \emph{area formula} (see for example \cite[lemma 2.91]{amfupa}), we have
\begin{equation}\label{change}
\haus(ND)=\int_{D} | N\tau (x)| d\haus(x),
\end{equation}
where $\tau(x)$ is the unit tangent vector to $D$ at the point $x\in D$.
Now, for $\haus$-a.e. $x\in D$,  $\tau=\tau(x)$ is orthogonal to the unit normal $\xi=\xi(x)$ that
appears in \eqref{hausm}: therefore, since $N^2=M^{-1}$,
we have
\begin{equation*}
|N\tau |^2=
\langle M^{-1} \tau,\tau\rangle
=\det M^{-1}
\langle \widehat{M} \tau,\tau\rangle
=\det M^{-1}
\langle M \xi,\xi\rangle
\end{equation*}
where $\widehat{M}$ is the cofactor matrix of $M$ (the last equality is typical of dimension two).
Taking square roots and integrating,
we find from \eqref{change} and the preceding formula
\begin{equation*}%\label{change2}
\haus(ND)
=\det M^{-1/2}
\int_{D} \sqrt{ \langle M \xi(x),\xi(x)\rangle}\,d\haus(x)
=\det M^{-1/2}\haus_M(D).
\end{equation*}
Moreover,
$|N\Om|=\det M^{-1/2}|\Om|$. Plugging the last two identities in \eqref{tmsecond},
we see that \eqref{estb2} is equivalent to \eqref{estb}.
\end{proof}

The upper bound \eqref{estb} will be at the basis of the $\Gamma$-liminf inequality
\eqref{gammainf}, and is therefore asymptotically optimal when
$\haus(D)\to\infty$ (see \cite{til} for a similar estimate for
 the compliance functional).

We shall also need the following result.

\begin{theorem}[The first eigenvalue of thin domains]\label{thin}
Let $E\subset\R^2$ be a bounded open set contained in a strip
\[
S=\{x\in\R^2\,|\,\, 0<\langle x,\xi\rangle<h\}
\]
of width $h>0$, where $\xi$ is the direction  orthogonal to $S$ ($|\xi=1|$).
If $M\in\R^{2\times 2}$ is positive definite, then
\begin{equation}\label{miracle3}
\lambda_1^M(E)>  \pi^2 \frac{\langle M\xi,\xi\rangle}{h^2}.
\end{equation}
\end{theorem}
\begin{proof}
Performing the linear change of variable
$x=M^{1/2}y$ in the integrals of the Rayleigh quotient,
the anisotropic eigenvalue $\lambda_1^M(E)$
is easily seen to coincide with the first Laplace eigenvalue of a new domain $\widehat E$, namely
\[
\lambda_1^M(E)=
\min_{\genfrac{}{}{0pt}{}{u\in H^1_0(E)}{u\not\equiv 0}}
\dfrac{\int_{E} \langle M\nabla u(x),\nabla u(x)\rangle\,dx}
{\int_{E} u(x)^2\,dx}
=
\min_{\genfrac{}{}{0pt}{}{v\in H^1_0(\widehat{E})}{v\not\equiv 0}}
\dfrac{\int_{\widehat{E}} |\nabla v(y)|^2\,dy}
{\int_{\widehat{E}} v(y)^2\,dy}
=
\lambda_1^I(\widehat E)
\]
where $\widehat E=M^{-1/2}E$ is now contained in the
new strip $\widehat S=M^{-1/2}S$. Observe that
\[
\widehat{S}
=\{y\in\R^2\,|\,\, 0<\langle y,M^{1/2}\xi\rangle<h\},
\]
so that the width of the strip $\widehat S$ is given by
\begin{equation}
\label{newwidth}
\widehat h=\dfrac {h}{\left|
M^{1/2}\xi\right|}
=
\dfrac {h}{\sqrt{
\langle M\xi,\xi\rangle}}.
\end{equation}
Therefore, since $\widehat{E}$ is bounded, it can be boxed in a thin rectangle $R$ of size
$\widehat{h}\times k$ for some $k>0$ (e.g. $k=\mathop{\rm diam}(\widehat{E})$). Then,
by monotonicity and the explicit expression for the first Laplace
eigenvalue
of a rectangle, we obtain
\[
\lambda_1^I(\widehat E)\geq \lambda_1^I(R)=
\dfrac{\pi^2}{\widehat{h}^2}+
\dfrac{\pi^2}{k^2}>
\dfrac{\pi^2}{\widehat{h}^2}
\]
which combined with \eqref{newwidth} proves \eqref{miracle3}.
\end{proof}

\section{The $\Gamma$-liminf inequality}

This section is devoted to proving that the $\Gamma$-liminf functional is
not smaller
 than the functional $F_\infty$ defined in \eqref{gammalimit}.

\begin{proposition}[$\Gamma$-liminf inequality]\label{prop8}
For every  varifold $\theta\in\prob$ and every sequence $\{\theta_L\}\subset\prob$ such that $\theta_L\weak\theta$, it holds
\begin{equation}\label{gammainf}
  \liminf_{L\to \infty} F_L(\theta_L)\geq F_\infty(\theta).
\end{equation}
\end{proposition}
\begin{proof}
Passing if necessary to a subsequence (not relabelled), we may assume that the liminf is
a \emph{finite} limit.  By \eqref{fl} this implies that,
for $L$ large enough, every $\theta_L$ is of the kind $\theta_L=\theta_{\Sigma_L}$
as defined in \eqref{defT}, for a suitable set $\Sigma_L\in\Alo$.
Therefore, by \eqref{defT} and \eqref{weakconvergenceT}, the weak-* convergence
$\theta_L\weak \theta$ takes the concrete form
\begin{equation}\label{varifold}
\lim_{L\to\infty}\frac{1}{\haus(\Si_L)}\int_{\Si_L} \varphi(x,\xi_{\Si_L}(x))d\haus(x)
=
\int_{\Ots} \varphi%(x,y)
\,d\theta\quad
\forall
\varphi\in\Csym,
\end{equation}
where $\xi_{\Sigma_L}$ is the unit normal to $\Sigma_L$.
Similarly,
using \eqref{fl} and \eqref{gammalimit},
the inequality in \eqref{gammainf} is, after taking square roots,
equivalent to
\begin{equation}
\label{toprove}
\lim_{L\to \infty}\frac{L}{\laa(\Omega\setminus\Sigma_L)^{1/2}}
\geq
\frac{1}{\pi}\esssup_{x\in\Omega}\frac{1}{f_\mu(x)\int_{\mathbb{S}^1} \langle
A(x)y,y\rangle\, d\nu_x(y)},
\end{equation}
where the function $f_\mu$ and the measures $\{\nu_x\}$  are as in
Definition~\ref{defFi}.

The finiteness of the
limit in \eqref{toprove} entails
that $\laa(\Omega\setminus\Sigma_L)\to\infty$, and this in turn forces
the sets $\Sigma_L$ to converge to $\cOmega$ in the Hausdorff metric
(otherwise, a subsequence among the open sets $\Omega\setminus\Sigma_L$ would contain a ball
$B_r(x_L)$
of fixed
radius $r>0$, and by monotonicity we would
have $\laa(\Omega\setminus\Sigma_L)\leq \laa(B_r(x_L))$: by \eqref{unifell}
this bound would be uniform in $L$, a contradiction).

Now choose an arbitrary open square $Q$ such that
$\overline{Q}\subset\Omega$.
Since $\Sigma_L$ is connected, the convergence $\Sigma_L\to\cOmega$ implies
(see e.g. \cite{mostil}) that
\begin{equation}\label{intersection}
\lim_{L\to \infty} \haus(\Si_L\cap Q)=\infty
\end{equation}
and forces $\Sigma_L$ to cross the boundary $\partial Q$ for large $L$, so that
\begin{equation}
\label{DLconn}
(\Sigma_L\cap Q)\cup \partial Q\qquad
\text{is connected (if $L$ is large enough).}
\end{equation}
Now fix a number $\eps>0$, and
consider the matrix $M_Q:=A(x_0)$,
where $x_0$ is the center of the square $Q$:
since $A(x)$ is uniformly continuous,
the conditions
\begin{equation}\label{mmm}
 \langle A(x)y,y\rangle \leq (1+\eps)^2  \langle M_Qy,y\rangle \qquad \forall x\in Q, \, \forall
 y\in\mathbb{S}^1
\end{equation}
are satisfied
as soon as the diameter of $Q$ is small enough (depending only on $\eps$).
Therefore, using the inequality $L\geq \haus(\Sigma_L)$, the  monotonicity of the first
eigenvalue with respect to domain inclusion, and \eqref{mmm},
it follows that
\begin{equation}\label{eq0}
\frac{L}{\laa(\Oms_L)^{1/2}}
\geq
\frac{\haus(\Si_L)}{\laa(Q\setminus\Sigma_L)^{1/2}}
\geq \frac{\haus(\Si_L)}{(1+\eps)\la_1^{M_Q}(Q\setminus \Si_L)^{1/2}},
\end{equation}
provided that, as we  shall assume, $\mathop{\rm diam}(Q)$ is small enough.

Now
$\lambda_1^{M_Q}(Q\setminus \Si_L)$
is the first eigenvalue in $H^1(Q)$, with $M_Q$ as coefficient matrix and
Dirichlet condition along $(\Sigma_L\cap Q)\cup \partial Q$: therefore, it
can be estimated using
Theorem~\ref{upper}, applied with
$\Omega=Q$, $M=M_Q$ and $D=(\Sigma_L\cap Q)\cup \partial Q$
 (by \eqref{DLconn}, we also have $\kappa=1$
if $L$ is large).  Then \eqref{estb}, taking square roots,
reads
\begin{equation}
\label{estb22}
\lambda_1^{M_Q}(Q\setminus\Sigma_L)^{1/2}\leq
\frac{{\pi}}{2t_L}\left(1+\frac{\pi t_L \det M_Q^{1/2}}{\ell_L}\right)^{1/2},
\end{equation}
where $t_L$, according to \eqref{deft}, is given by
\begin{equation}\label{deftL}
 t_L=\dfrac{|Q|}{\ell_L +\sqrt{\ell_L^2+\pi |Q|\det {M_Q}^{1/2}}},
\end{equation}
while $\ell_L$, according to \eqref{hausm}, is the Riemannian length
\begin{equation}
\label{deflL}
\ell_L=\int_{(\Sigma_L\cap Q)\cup \partial Q} \sqrt{\langle M_Q\xi_L(x),\xi_L(x)\rangle}\,  d\haus(x),
\end{equation}
$\xi_L(x)$ being the unit normal to $(\Sigma_L\cap Q)\cup \partial Q$. We shall let
$L\to\infty$ in \eqref{eq0} and use \eqref{estb22},  hence
we are interested in
the asymptotics (as $L\to\infty$)
of both $t_L$ and $\ell_L$. As $M_Q=A(x_0)$,
the last equation and \eqref{unifell}
give
\[
\ell_L\geq
C^{-1/2}\haus\bigl((\Sigma_L\cap Q)\cup \partial Q\bigr)
>
C^{-1/2}\haus(\Sigma_L\cap Q),
\]
so that $\ell_L\to\infty$  by \eqref{intersection}.
Therefore, from \eqref{deftL} we find the asymptotics
\begin{equation}
\label{asyt}
 t_L
 \sim \dfrac{|Q|}{2\ell_L }\quad\text{as $L\to\infty$}
 \end{equation}
so that, letting $L\to\infty$ in \eqref{eq0},
and using \eqref{estb22}, \eqref{asyt}, we obtain
\begin{equation}\label{eq00}
\begin{split}
\lim_{L\to\infty}\frac{L}{\laa(\Oms_L)^{1/2}}
&\geq
\frac 1{1+\eps}\,\liminf_{L\to\infty}
\frac {\haus(\Si_L) \,2 t_L}{\pi}\\
&=
\frac {|Q|}{(1+\eps)\pi}\,\liminf_{L\to\infty}
\frac {\haus(\Si_L) }{\ell_L}.
\end{split}
\end{equation}
To estimate the last liminf, we get back to \eqref{deflL} and observe that
\begin{equation}
\label{asylL}
\ell_L\sim
\int_{\Sigma_L\cap Q} \sqrt{\langle M_Q\xi_{\Sigma_L}(x),\xi_{\Sigma_L}(x)\rangle}\,  d\haus(x)\quad
\text{as $L\to\infty$,}
\end{equation}
since the contribution of $\partial Q$ to the integral in \eqref{deflL} is fixed,
while that of $\Sigma_L\cap Q$ is dominant by \eqref{intersection}.
Now if $\eta\in C(\cOmega)$ is a cutoff function such that $0\leq \eta\leq 1$
and $\eta\equiv 1$ over $\overline{Q}$, letting
\begin{equation}
\label{defphi}
\widehat\varphi(x,y):=\eta(x)\sqrt{\langle M_Q y, y\rangle},\quad x\in\cOmega,\quad
y\in\sfera,
\end{equation}
we clearly have $\varphi\in\Csym$. Therefore, using \eqref{asylL} and
\eqref{varifold} we infer that
\[
\begin{split}
&\limsup_{L\to\infty}
\frac
{\ell_L} {\haus(\Si_L) }
=
\limsup_{L\to\infty}
\frac
1 {\haus(\Si_L) }
\int_{\Sigma_L\cap Q} \sqrt{\langle M_Q\xi_{\Sigma_L}(x),\xi_{\Sigma_L}(x)\rangle}\,  d\haus(x)
\\
\leq
&\lim_{L\to\infty}
\frac
1 {\haus(\Si_L) }
\int_{\Sigma_L} \widehat\varphi(x,\xi_{\Sigma_L}(x))\,  d\haus(x)
=
\int_{\Ots} \widehat\varphi(x,y)%(x,y)
\,d\theta.
\end{split}
\]
Now, recalling the properties of the cutoff function $\eta$ in \eqref{defphi},
it is possible to let $\eta(x)\downarrow \chi_{\overline Q}(x)$ pointwise in the last integral,
and obtain by dominated convergence
\[
\limsup_{L\to\infty}
\frac
{\ell_L} {\haus(\Si_L) }
\leq
\int_{\overline{Q}\times \sfera} \sqrt{\langle M_Q y, y\rangle}
\,d\theta
=
\int_{\overline{Q}}\left( \int_{\sfera} \sqrt{\langle M_Q y, y\rangle}\,d\nu_x(y)
\right)
\,d\mu(x),
\]
where $\theta=\mu\otimes \nu_x$ is the slicing of $\theta$ as in Definition~\ref{defFi}.
Passing to reciprocals in the previous inequality, we obtain a \emph{lower} bound
for the last liminf in \eqref{eq00} which, plugged into \eqref{eq00}, yields
\begin{equation}
\label{finales}
\lim_{L\to\infty}\frac{L}{\laa(\Oms_L)^{1/2}}
\geq
\frac {|Q|}{(1+\eps)\pi \int_{\overline{Q}}\left( \int_{\sfera} \sqrt{\langle M_Q y, y\rangle}\,d\nu_x(y)
\right)
\,d\mu(x)}.
\end{equation}
This lower bound holds for every square $Q$ such that $\overline{Q}\subset\Omega$, and having a
sufficiently small side length (depending only on $\eps$, as to guarantee the validity of \eqref{mmm}).
Moreover, the matrix $M_Q=A(x_0)$ where $x_0$ is the center of $Q$, so that $M_Q$ is independent of
the side length of $Q$. Thus, given $\eps>0$, we can first choose an arbitrary $x_0\in\Omega$, and
then rely on the previous inequality for every $Q$ (of sufficiently
small side length) centered at $x_0$: it is therefore possible, for fixed $x_0$, to let $Q$ shrink
around $x_0$. On the other hand, from Lebesgue Differentiation Theorem we have
\[
\lim_{Q\downarrow x_0} \frac {1}{|Q|}
\int_{\overline{Q}}\left( \int_{\sfera} \sqrt{\langle M_Q y, y\rangle}\,d\nu_x(y)
\right)
\,d\mu(x)=
f_\mu(x_0) \int_{\sfera} \sqrt{\langle A(x_0) y, y\rangle}\,d\nu_x(y)
%\quad
%\text{for a.e. $x_0\in\Omega$,}
\]
for a.e. $x_0\in\Omega$ (w.r.to the Lebesgue measure), where $f_\mu$ is as in Definition~\ref{defFi}.
Thus, to obtain \eqref{toprove}, we can first shrink $Q$ around its center $x_0$ in \eqref{finales},
then take the essential supremum over $x_0\in\Omega$ on the right, and finally use the arbitrariness of
$\eps$.
\end{proof}

\section{The $\Gamma$-limsup inequality}

In the following proposition we construct a fundamental pattern, whose
periodic homogenization inside $\Omega$ will be the main ingredient
in the construction of the recovery sequence, for the $\Gamma$-limsup inequality.

\begin{proposition}[tile construction for elementary varifolds]\label{proptile}
Let $Q\subset\R^2$ be an open
square with sides parallel to the coordinate axes, and
let $\nu$ be a probability measure over ${\mathbb S}^1$.
Moreover, let $M\in\R^{2\times 2}$ be a positive definite
matrix.
Then, for every length $\ell$ large enough, there exists a continuum $\Sigma_\ell\in
{\mathcal A}_\ell(Q)$ with the following properties.
\begin{itemize}
\item [(i)] Boundary-covering and length matching properties:
\begin{equation}
\label{bclm}
\partial Q\subset\Sigma_\ell,
\quad\text{and}\quad
 \lim_{\ell\to \infty}\frac{\haus(\Sigma_\ell)}{\ell}=1.
\end{equation}
\item [(ii)] Estimate on the anisotropic eigenvalue:
\begin{equation}\label{stimabase}
 \limsup_{\ell\to \infty} \dfrac
{\ell^2}{\lambda_1^M(Q\setminus\Sigma_\ell)}
\leq \frac{|Q|^2}{\big(\pi\int_{\mathbb{S}^1}
\sqrt{ \langle My,y\rangle}\, d\nu(y) \big)^2}.
\end{equation}
\item[(iii)] The varifolds associated with $\Sigma_\ell$ converge
to $|Q|^{-1}\chi_Q \otimes \nu$, that is,
for
every function $\varphi\in C(\overline{Q}\times {\mathbb S}^1)$ such that
$\varphi(x,y)=\varphi(x,-y)$, one has
\begin{equation}
\label{varifoldsconverge} \lim_{\ell\to\infty} \frac 1{\haus(\Sigma_\ell)}
\int_{\Sigma_\ell} \varphi(x,\xi_\ell(x))\,d\haus(x) = \frac
1 {|Q|}\int_Q\int_{{\mathbb
S}^1}\varphi(x,y)\,d\nu(y)\,dx
\end{equation}
where $\xi_\ell$ is any measurable selection of the unit normal to
$\Sigma_\ell$.
\end{itemize}
\end{proposition}
\begin{proof}
It is not restrictive to assume that $Q$
has a side length of $1$ (the general case, with a side length of $s$,
is recovered by a scaling argument, replacing $Q$ with $sQ$ and the constructed $\Sigma_\ell$
with the rescaled $s\Sigma_{\ell/s}$).

Moreover, we initially assume that
the measure $\nu$ is \emph{purely atomic}, that is
\begin{equation}
\label{purelyatomic}
\nu=\sum_{j=1}^n \beta_j \delta_{\xi_j},\quad \sum_{j=1}^n \beta_j=1\qquad\text{($n\geq 1$)}
\end{equation}
for suitable %$n\geq 1$,
weights $\beta_j>0$ and unit vectors $\xi_j\in {\mathbb S}^1$.
%(this restriction will be removed later).
With this notation,
we have
\begin{equation}
\label{defI}
I:=\int_{{\mathbb S}^1}\sqrt{\langle M\xi,\xi\rangle}\,d\nu(\xi)
=
\sum_{j=1}^n \beta_j
\sqrt{\langle M\xi_j,\xi_j\rangle}
\end{equation}
(the first equality is the definition of $I$)
and, more generally,
\begin{equation}
\label{moregen}
\int_{{\mathbb S}^1}v(\xi)\,d\nu(\xi)=
\sum_{j=1}^n \beta_j v(\xi_j)\quad
\forall v\in C({\mathbb S}^1).
\end{equation}
The sets $\Sigma_\ell$ will be obtained as the periodic homogenization,
inside the square $Q$, of suitably rescaled fundamental tiles $\Gamma_\eps$,
initially constructed inside a unit square.
Our construction consists of three main steps.

\medskip

\noindent\emph{Step 1: tile construction}. %(See Figure~\ref{fig1}).
We first slice the unit square $Y:=(0,1)\times (0,1)$ into $n$
stacked rectangles $Y_j$
($1\leq j\leq n$) of size $1\times h_j$,
the height $h_j$ being defined as
\begin{equation}
\label{defhj}
h_j:=\dfrac{\beta_j \sqrt{\langle M\xi_j,\xi_j\rangle}
}{I}.
\end{equation}
Note that, by \eqref{defI}, $\sum h_j=1$ so that the heights of
the $n$ rectangles match the height of $Y$.
%: the explicit form of the $h_j$s, however, will not be used immediately).
Then, by
drawing inside every $Y_j$
a maximal family
of parallel line segments, orthogonal to $\xi_j$ and equally spaced a
distance  of $\eps_j$ apart from one another ($\eps_j\ll 1$ to be chosen later), we
further slice every rectangle $Y_j$ into several thin polygons % $S_{j,k}$
of
width $\eps_j$. For fixed $j$,
the number of polygons inside $Y_j$ is $O(1/\eps_j)$:
every polygon
is a
trapezoid of height $\eps_j$, with the exception of \emph{at most four
polygons} that, due to
 the corners of $Y_j$,
may degenerate into a triangle
 or a pentagon (one hexagon may also occur,
 if the line segments are almost parallel
to a diagonal of $Y_j$ and  one polygon touches two opposite
corners of $Y_j$).

Let $K_j$ denote the union of all these line segments, orthogonal to $\xi_j$,
drawn inside $Y_j$. Since
 every polygon (with at most four exceptions)
is a trapezoid of height $\eps_j$, summing the areas of the polygons we see that
\[
\eps_j \haus(K_j)+O(\eps_j)=|Y_j|=h_j,
\]
where $O(\eps_j)$ is the correction due to the exceptional polygons, whose total
area is at most $4\sqrt{2}\eps_j$. Thus, for the total length of $K_j$ we
have the asymptotics
\begin{equation*}
%\label{asyKj}
\haus(K_j)\sim \frac{h_j}{\eps_j}\quad\text{as $\eps_j\to 0.$}
\end{equation*}
From now on, the values of the $\eps_j$s shall be fixed according to
\begin{equation}
\label{defepsj}
\eps_j:=\eps \sqrt{\langle M\xi_j,\xi_j\rangle},\quad
1\leq j\leq n
\end{equation}
(where $\eps\ll 1$ is a scale parameter to be tuned later)
so that,
using \eqref{defhj},
the previous asymptotics take the concrete form
\begin{equation}
\label{asyKj}
\haus(K_j)\sim \frac{\beta_j}{\eps I}\quad\text{as $\eps\to 0.$}
\end{equation}
Our fundamental object, the \emph{tile} $\Gamma_\eps$,
is defined as
\begin{equation}
\label{defGammaeps}
\Gamma_\eps:=R\cup S_\eps,\quad
R:=\bigcup_{j=1}^n \partial Y_{j},\quad
S_\eps:=\bigcup_{j=1}^n K_{j}.
\end{equation}
The set $R$ (which is \emph{independent} of $\eps$)
consists of the boundaries of the $n$ rectangles $Y_j$ and acts as a frame, while
$S_\eps$ is the union of all the oblique line segments inside the rectangles.
Clearly $\Gamma_\eps$ is compact and
connected, and moreover
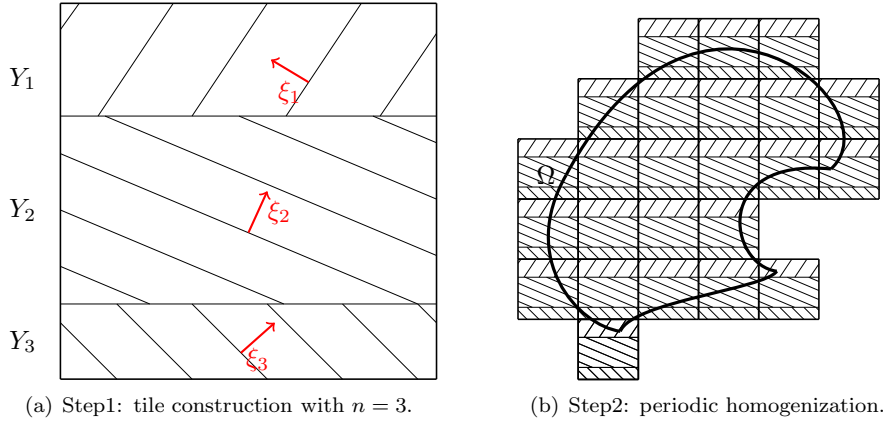
\begin{equation}
\label{frame}
\partial Y\subset \Gamma_\eps
\end{equation}
since $\partial Y\subset R$. We also have
\begin{equation}
\label{LRS}
\haus(R)=n+3,\qquad
\haus(S_\eps)  \sim \frac 1{\eps I}\quad\text{as $\eps\to 0$}
\end{equation}
having used, in the last expansion, \eqref{asyKj}
and the fact that $\sum \beta_j=1$.

In view of \eqref{varifoldsconverge}, let
$\xi_\eps(x)$ denote the unit normal to the set $S_\eps$
at $x\in S_\eps$ (as usual, any measurable selection of $\pm \xi_\eps(x)$ will do).
By our construction,  $\xi_\eps(x)=\pm \xi_j$ for every $x\in K_j$; therefore,
if $v\in C({\mathbb S}^1)$ is such that $v(y)=v(-y)$, we have
\[
\frac {1}{\haus(S_\eps)}
\int_{S_\eps} v(\xi_\eps(t))\,d\haus(t)
=
\sum_{j=1}^n
\frac {\haus(K_j)}{\haus(S_\eps)}
v(\xi_j).
\]
Therefore, using \eqref{asyKj} and \eqref{LRS} we have
\begin{equation}
\label{limIv}
\lim_{\eps\to 0}
\frac {1}{\haus(S_\eps)}
\int_{S_\eps} v(\xi_\eps(t))\,d\haus(t)
=
\sum_{j=1}^n
\beta_j v(\xi_j)=\int_{{\mathbb S}^1}v(\xi)\,d\nu(\xi)
\end{equation}
by \eqref{moregen}.
Finally, since every connected component of
$Y\setminus \Gamma_\eps$ is, by construction, a polygon of width $\eps_j$
in the direction $\xi_j$ for some $j\in\{1,\ldots,n\}$, Theorem~\ref{thin}
 gives
\begin{equation}
\label{ss6}
\lambda_1^M(Y\setminus\Gamma_\eps)\geq \min_{j}
\pi^2 \frac {\langle M \xi_j,\xi_j\rangle}{\eps_j^2}=
\frac {\pi^2}{\eps^2}
\end{equation}
having used \eqref{defepsj} in the last passage.

\medskip

\begin{figure}
\centering

\subfigure[Step1: tile construction with $n=3$.]
{
\begin{tikzpicture}[scale=0.5]

\rectangle
\draw[->,thick,red] (6.6,7.9)--(5.6,8.5) node[midway, below] {$\xi_1$} ; %VETTORI
\draw[->,thick,red] (5,3.9)--(5.5,5) node[midway,right] {$\xi_2$};
\draw[->,thick,red] (4.8,0.7) --(5.7,1.5) node[midway, below] {$\xi_3$} ;
\node at (-1,8) {$Y_1$} ;
\node at (-1,4.5) {$Y_2$} ;
\node at (-1,1) {$Y_3$} ;
\end{tikzpicture}
}
\qquad
\subfigure[Step2: periodic homogenization.]{
\begin{tikzpicture}[scale=0.08]

\draw[very thick, rounded corners] (17,8) .. controls (10,9) and (0,20) .. (8,34) node [left] {$\Omega$}  ;
\draw[very thick, rounded corners] (8,34) .. controls (30,75) and (63,45) .. (52,35);
\draw[very thick, rounded corners] (52,35) .. controls (32,37) and (35,19) .. (43,18);
\draw[very thick, rounded corners] (43,18) .. controls +(-2,-3) and +(2,5) .. (17,8);

%\rectangle

\begin{scope}[shift={(10,0)}]
\rectangle
\end{scope}

\begin{scope}[shift={(10,0)}]
\rectangle
\end{scope}

%\begin{scope}[shift={(20,0)}]
%\rectangle
%\end{scope}

%\begin{scope}[shift={(30,0)}]
%\rectangle
%\end{scope}

%\begin{scope}[shift={(40,0)}]
%\rectangle
%\end{scope}

%\begin{scope}[shift={(50,0)}]
%\rectangle
%\end{scope}

\begin{scope}[shift={(0,10)}]
\rectangle
\end{scope}

\begin{scope}[shift={(10,10)}]
\rectangle
\end{scope}

\begin{scope}[shift={(20,10)}]
\rectangle
\end{scope}

\begin{scope}[shift={(30,10)}]
\rectangle
\end{scope}

\begin{scope}[shift={(40,10)}]
\rectangle
\end{scope}

%\begin{scope}[shift={(50,10)}]
%\rectangle
%\end{scope}

\begin{scope}[shift={(0,20)}]
\rectangle
\end{scope}

\begin{scope}[shift={(10,20)}]
\rectangle
\end{scope}

\begin{scope}[shift={(20,20)}]
\rectangle
\end{scope}

\begin{scope}[shift={(30,20)}]
\rectangle
\end{scope}

%\begin{scope}[shift={(40,20)}]
%\rectangle
%\end{scope}

%\begin{scope}[shift={(50,20)}]
%\rectangle
%\end{scope}

\begin{scope}[shift={(0,30)}]
\rectangle
\end{scope}

\begin{scope}[shift={(10,30)}]
\rectangle
\end{scope}

\begin{scope}[shift={(20,30)}]
\rectangle
\end{scope}

\begin{scope}[shift={(30,30)}]
\rectangle
\end{scope}

\begin{scope}[shift={(40,30)}]
\rectangle
\end{scope}

\begin{scope}[shift={(50,30)}]
\rectangle
\end{scope}

%\begin{scope}[shift={(0,40)}]
%\rectangle
%\end{scope}

\begin{scope}[shift={(10,40)}]
\rectangle
\end{scope}

\begin{scope}[shift={(20,40)}]
\rectangle
\end{scope}

\begin{scope}[shift={(30,40)}]
\rectangle
\end{scope}

\begin{scope}[shift={(40,40)}]
\rectangle
\end{scope}

\begin{scope}[shift={(50,40)}]
\rectangle
\end{scope}

%\begin{scope}[shift={(0,50)}]
%\rectangle
%\end{scope}

%\begin{scope}[shift={(10,50)}]
%\rectangle
%\end{scope}

\begin{scope}[shift={(20,50)}]
\rectangle
\end{scope}

\begin{scope}[shift={(30,50)}]
\rectangle
\end{scope}

\begin{scope}[shift={(40,50)}]
\rectangle
\end{scope}

%\begin{scope}[shift={(50,50)}]
%\rectangle
%\end{scope}

\end{tikzpicture}
}
\end{figure}

\noindent\emph{Step 2 (periodic homogenization)}. Since the square $Q$
has a side length of $1$,  given an integer $m\geq 1$ we may fit $m^2$
copies of the rescaled tile $m^{-1}\Gamma_\eps$ inside $\overline Q$
as in an $m\times m$ checkerboard:
the resulting
tiling is then $1/m$--periodic in the two directions parallel to the sides of $Q$.

We denote by $\Gamma_{m,\eps}$
the union of these $m^2$ rescaled tiles: this set is connected, because so is $\Gamma_\eps$
and,  by \eqref{frame}, each tile shares
a side of length $1/m$ with each neighbor.
The sets $\Sigma_\ell$ we want to construct are defined,
for large $\ell$, as follows:
\begin{equation*}
%\label{defSigmaL}
\Sigma_\ell:=\Gamma_{m,\eps}\qquad\text{$m=m(\ell)$ and $\eps=\eps(\ell)$,}
\end{equation*}
where
\begin{equation}
\label{defepsm}
\eps(\ell):=\ell^{-2/3},
\qquad
m(\ell)=\left\lceil \ell^{\frac 1 3}I \right\rceil
\end{equation}
(observe that
$\eps\to 0$ and $m\to\infty$
when $\ell\to\infty$).

According to \eqref{defGammaeps},
$\Sigma_\ell$ is the union
of $m^2$ disjoint copies of the rescaled set $m^{-1}S_\eps$ (hereafter denoted by
$S_{m,\eps}^{i}$, $1\leq i\leq m^2$) of total length
$\sim m/(\eps I)$ according to \eqref{LRS},
plus the union of $m^2$
copies of the rescaled frame $m^{-1}R$: these are not disjoint, since
adjacent tiles share a side, but their total length does not exceed $m(n+3)$
according to \eqref{LRS}, and is therefore negligible compared
to the total length of the $S_{m,\eps}^i$s. As a consequence,
\begin{equation}
\label{estLL}
\haus(\Sigma_\ell)\sim \haus\left(\bigcup\nolimits_{i=1}^{m^2} S_{m,\eps}^i\right)
\sim \frac{m}{\eps I}\quad\text{as $\ell\to\infty$,}
\end{equation}
and \eqref{bclm} follows from \eqref{defepsm} (the first part of \eqref{bclm} is
immediate, hence claim (i) is proved).

Similarly, the limit in \eqref{varifoldsconverge} is unchanged
if we restrict integration to the union of the
$S_{m,\eps}^i$s,
and therefore
\eqref{varifoldsconverge} is equivalent to
\begin{equation*}
%\label{varifoldsconverge2}
\lim_{\ell\to\infty}
\sum_{i=1}^{m^2}
\frac 1 {m^2 \haus(S^i_{m,\eps})}
\int_{S_{m,\eps}^i} \varphi(x,\xi_\ell(x))\,d\haus(x) = \frac
1 {|Q|}\int_Q\int_{{\mathbb
S}^1}\varphi(x,y)\,d\nu(y)\,dx,
\end{equation*}
having used the
first equivalence in \eqref{estLL} and the fact that the
$S_{m,\eps}^i$s are disjoint and congruent.
In addition,
by a density argument,
it suffices to prove \eqref{varifoldsconverge} when
$\varphi(x,y)=u(x)v(y)$ with $u\in C(\overline{Q})$ and $v\in C({\mathbb S}^1)$
such that $v(y)=v(-y)$, so that
the last equation reduces to
\begin{equation}
\label{varifoldsconverge3} \lim_{\ell\to\infty}
\sum_{i=1}^{m^2}
\frac 1 {m^2 \haus(S^i_{m,\eps})}
\int_{S_{m,\eps}^i} u(x)v(\xi_\ell(x))\,d\haus(x) =
I_u I_v,
\end{equation}
where
\[
I_u=
\int_Q u(x) \,dx,\qquad
I_v=\int_{{\mathbb
S}^1}
v(y)\,d\nu(y)
\]
(as $1/|Q|=1$, this factor  can be omitted).
Now fix functions $u,v$ as above, and let
\begin{equation}
\label{defUmi}
U_m^i :=\frac
1 {|Q^i_m|}\int_{Q^i_m} u(x) \,dx =
m^2 \int_{Q^i_m} u(x) \,dx
\end{equation}
denote the integral average of $u$ over the square $Q^i_m$, of side-length $1/m$,
whose boundary frames $S_{m,\eps}^i$. If $\omega(\delta)=\sup_{|z-x|\leq \delta} |u(z)-u(x)|$
is the modulus of continuity of $u$ over $\overline{Q}$,
since $S_{m,\eps}^i\subset \overline{Q^i_m}$ and $\mathop{\rm diam}(Q^i_m)=\sqrt{2}/m$
we have
\[
\frac 1 {\haus(S_{m,\eps}^i)}
\int_{S_{m,\eps}^i}
\left\vert u(x)-U^i_m
\right\vert
\,d\haus(x)
\leq \omega(\sqrt{2}/m),\quad
1\leq i\leq m^2.
\]
From these estimates, we immediately see that
\[
\left\vert
\sum_{i=1}^{m^2}
\frac 1 {m^2 \haus(S^i_{m,\eps})}
\int_{S_{m,\eps}^i} \bigl( u(x)-U^i_m\bigr) v(\xi_\ell(x))\,d\haus(x)
\right\vert
\leq \omega(\sqrt{2}/m)\Vert v\Vert_{L^\infty},
\]
which tends to zero as $\ell$ (hence $m$) tends to infinity.
Thus, to prove \eqref{varifoldsconverge3}, we can freeze $u(x)$
in each integral, and replace it with
the constant $U^i_m$. On the other hand,
since each $S_{m,\eps}^i$ is a rescaled copy of the set $S_\eps$,
\[
\frac {1} {\haus(S^i_{m,\eps})}
\int_{S_{m,\eps}^i} v(\xi_\ell(x))\,d\haus(x)
=
\frac {1} {\haus(S_{\eps})}
\int_{S_\eps} v(\xi_\eps(x))\,d\haus(x),\quad
1\leq i\leq m^2,
\]
where $\xi_\eps$, as before, is the unit normal to $S_\eps$.
Therefore, if we replace $u(x)$ with $U_m^i$ in \eqref{varifoldsconverge3}
and take it out of the integral,
\eqref{varifoldsconverge3} simplifies to
\[
\lim_{\ell\to\infty}
\left(
\sum_{i=1}^{m^2}
\frac {1} {m^2} U^i_m
\right)
\left(\frac {1} {\haus(S_{\eps})}
\int_{S_\eps} v(\xi_\eps(x))\,d\haus(x)\right)=I_uI_v
\]
which is now obvious, since the sum coincides with $I_u$
by \eqref{defUmi}, while the second factor tends to $I_v$ by \eqref{limIv}.
Summing up, \eqref{varifoldsconverge} is proved.

Finally, to prove \eqref{stimabase}, observe that $Q\setminus\Sigma_\ell$
is highly disconnected due to the checkerboard
structure of $\Sigma_\ell$, and each of its connected component
is, by construction, congruent to a connected component of $Y\setminus\Gamma_\eps$ scaled
down by a factor $1/m$ (which amplifies the first eigenvalue by
a factor $m^2$). Therefore, we have using \eqref{ss6} and \eqref{defepsm}
\[
\lambda_1^M(Q\setminus \Sigma_\ell)=m^2
\lambda_1^M(Y\setminus \Gamma_\eps)\geq
\frac{\pi^2m^2}{\eps^2}\sim \pi^2 I^2 \ell^2
\quad\text{as $\ell\to\infty$,}
\]
so that \eqref{stimabase} follows immediately from \eqref{defI}.

\medskip

\noindent\emph{Step 3 (general $\nu$)}.
By a diagonal argument, we can now remove the restriction \eqref{purelyatomic}
and consider an arbitrary probability measure
$\nu$ over ${\mathbb S}^1$.

Let $\{\nu_k\}$ be a sequence of atomic probability
measures over ${\mathbb S}^1$, such that $\nu_k\rightharpoonup \nu$.
For each $k$, we can apply Proposition~\ref{proptile} in the
form just proved (with $\nu_k$
in place of $\nu$),
thus obtaining
continua
$\Sigma_\ell^k\in\mathcal{A}_{\ell}(Q)$ satisfying claims (i)--(iii) relative
to $\nu_k$.
Then, since the weak-* topology of
probability
measures over ${\mathbb S}^1$ is metrizable,
by a standard diagonal argument one can find an increasing sequence
of lengths $\ell_1<\ell_2<\cdots $ such that, letting $\Sigma_\ell:=\Sigma_\ell^k$
whenever $\ell_k\leq \ell< \ell_{k+1}$, the resulting continua $\{\Sigma_\ell\}$
satisfy claims (i)--(iii) relative to $\nu$.
\end{proof}

In view of extending Proposition~\ref{proptile} to a more
general class of varifolds, the following terminology is needed.

\begin{definition}\label{defstep}
For $s>0$, let $\chb$ denote the collection of all those open squares $Q_i\subset\R^2$, with side-length $s$
and corners
on the lattice $(s{\mathbb Z})^2$, such that $Q_i\cap \Omega\not=\emptyset$.
We say that a varifold  $\theta\in\prob$ is \emph{fitted to $\chb$} if
%$\mu$ is absolutely continuous, with a density $f(x)>0$ which is constant on each set of the form $Q_i\cap \Omega$ with $Q_i\in\chb$ and $\nu_x$ is an atomic probability measure on $\mathbb S_1$ concentrated at a finite number of points, which is even and does not depend on $x$ when $x$ ranges in $Q_i\cap \Omega$, with $Q_i\in\chb$.
it can be represented as
\begin{equation}\label{mufit}
\theta=\sum_{Q_i\in \chb} \alpha_i \chi_{\Omega\cap Q_i} \otimes \nu_i
%alpha_i \chi_{\Omega_i}(x)dx\otimes \sum_{j=1}^{2n_i} \gammaloc d\delta_{\xiloc}\big)\quad \Omega_i=Q_i\cap \Om,
\end{equation}
for suitable constants $\alpha_i\geq 0$ and probability measures
$\nu_i$ over ${\mathbb S^1}$, satisfying
\begin{equation}\label{constants}
\sum_{Q_i\in\chb} \alpha_i|\Omega\cap Q_i| = 1.
\end{equation}
\end{definition}
Roughly speaking, a varifold is \emph{fitted to $\chb$} if its restriction to a cylinder of
the form $\overline{\Omega\cap Q_i}\times {\mathbb S^1}$ is the \emph{product measure} $\alpha_i\otimes \nu_i$.
In particular, for such a varifold its first marginal (i.e. its projection over $\overline{\Omega}$)
is absolutely continuous with respect to the Lebesgue measure, and is \emph{piecewise constant} over
$\Omega$: as a consequence, \eqref{constants} is equivalent to the request that the varifold has unit mass.

We observe, for future reference, that if $\theta$ is as in \eqref{mufit} the functional
$F_\infty$ defined in \eqref{gammalimit} takes the form
\begin{equation}
\label{Finffitted}
F_\infty(\theta)=
\max_{Q_i\in\chb}\sup_{x\in \Omega\cap Q_i}
\frac{1}{\alpha_i^2\big(\pi\int_{\mathbb{S}^1}
\sqrt{ \langle A(x)y,y\rangle}\, d\nu_i(y) \big)^2}.
\end{equation}

\begin{proposition}\label{proptile2}
Let $\theta\in\prob$ be a varifold \emph{fitted to $\chb$}.
%of the form \eqref{mufit}.
Then there exists a sequence of continua
$\Sigma_L\subset\overline{\Omega}$ such that \eqref{varifold} holds true,
\begin{equation}
\label{asL}
\lim_{L\to\infty}
\frac{\haus(\Sigma_L)}{L}=1,
\end{equation}
and
\begin{equation}\label{gammasup}
 \limsup_{L\to \infty}
\frac{L^2}
{\lambda_1^A(\Omega\setminus\Sigma_L)}\leq
% F_L(\nu_L)\leq \frac{1}{\pi^2}\esssup_{x\in\Omega}\frac{1}{\big(\int_{\mathbb{S}^1} |\sqrt{A(x)}y| d\nu_x(y) f(x)\big)^2}.
F_\infty(\theta).
\end{equation}
\end{proposition}
\begin{proof}
We keep for $\theta$ the same notation as in Definition~\ref{defstep}
(in particular \eqref{mufit} and \eqref{constants}),
%consider a varifold
%\begin{equation}\label{mufit2}
%\nu=\sum_{Q_i\in \chb} \alpha_i \chi_{\Omega\cap Q_i} \otimes \nu_i
%\end{equation}
%fitted to $\chb$
and fix a number
$\eta>1$.
By replacing $s$
with $s/2^n$ for some large $n\geq 1$ and
relabelling the $\alpha_i$s
(thus keeping $\nu$ fitted to $\chb$), we may assume that $s$ is so small that
the following two conditions hold:
\begin{itemize}
\item[(C1)] no connected component of $\partial\Omega$ is strictly contained in any square $Q_i\in\chb$;
\item[(C2)] for every square $Q_i\in\chb$, there exists a positive definite matrix $M_i$ such that
\begin{equation}\label{assu2}
\langle A(x) y,y\rangle\geq\frac{1}{\eta} \langle M_iy,y\rangle \qquad \forall x\in \Omega\cap Q_i,
\quad\forall y\in\mathbb S_1.
\end{equation}
\end{itemize}
Condition C1 holds, for $s$ small enough, since $\partial \Omega$, being Lipschitz,
has finitely many connected components whose diameters have a positive lower bound.
On the other hand, C2 is guaranteed, for small $s$, by the uniform continuity of the function $A(x)$
(one can set e.g.  $M_i=A(x_i)$, for some $x_i\in \Omega\cap Q_i$).

The sets $\Sigma_L$ we want to construct will be obtained as a patchwork of sets $\Sigma^i_\ell$,
one for each square $Q_i\in\chb$, obtained from Proposition~\ref{proptile}. More precisely,
for every square $Q_i\in\chb$ we apply Proposition~\ref{proptile} when the square $Q=Q_i$,
the measure $\nu=\nu_i$ and the matrix $M=M_i$. This yields, for
large enough $\ell$, sets $\Sigma^i_\ell$ such that:
\begin{itemize}
\item[(i)] $\partial Q_i\subset\Sigma^i_\ell\quad$ and $\quad\haus(\Sigma^i_\ell)\sim\ell$ as $\ell\to\infty$;
\item[(ii)]  $\displaystyle \limsup_{\ell\to \infty} \dfrac
{\ell^2}{\lambda_1^{M_i}(Q_i\setminus\Sigma^i_\ell)}
\leq \frac{s^4}{\big(\pi\int_{\mathbb{S}^1}
\sqrt{ \langle M_iy,y\rangle}\, d\nu_i(y) \big)^2};$
\item[(iii)] for
every $\varphi\in C(\overline{Q_i}\times {\mathbb S}^1)$ with
$\varphi(x,y)=\varphi(x,-y)$,
\begin{equation}
\label{varifoldsconverge4} \lim_{\ell\to\infty} \frac 1{\haus(\Sigma^i_\ell)}
\int_{\Sigma^i_\ell} \varphi(x,\xi^i_\ell(x))\,d\haus(x) = \frac
1 {|Q_i|}\int_{Q_i}\int_{{\mathbb
S}^1}\varphi(x,y)\,d\nu_i(y)\,dx
\end{equation}
where $\xi^i_\ell$ is any measurable selection of the unit normal to
$\Sigma^i_\ell$.
\end{itemize}
Observe that the domain $\Omega$ plays no role in this construction, and this is natural
for those squares $Q_i\in\chb$ such that $Q_i\subset\Omega$. If, however, $Q_i\in\chb$
is such that $Q_i\cap\partial\Omega\not=\emptyset$, the weak-* convergence in (iii)
(that occurs in the whole $\overline{Q_i}\times {\mathbb S}^1$)
can still
be \emph{localized} to $\overline{\Omega\cap Q_i}\times {\mathbb S}^1$. More precisely,
since $\partial(\Omega\cap Q_i)$ has null Lebesgue measure, and
the limit measure $|Q_i|^{-1}\chi_{Q_i}\otimes \nu_i$ does not
charge the cylinder $\partial(\Omega\cap Q_i)\times {\mathbb S}^1$, from \eqref{varifoldsconverge4}
we infer that
\begin{equation}
\label{varifoldsconverge5} \lim_{\ell\to\infty} \frac 1{\haus(\Sigma^i_\ell)}
\int_{\overline{\Omega}\cap \Sigma^i_\ell} \varphi(x,\xi^i_\ell(x))\,d\haus(x) = \frac
1 {|Q_i|}\int_{\Omega\cap Q_i}\int_{{\mathbb
S}^1}\varphi(x,y)\,d\nu_i(y)\,dx,
\end{equation}
for
every $\varphi\in C(\overline{\Omega\cap Q_i}\times {\mathbb S}^1)$ with
$\varphi(x,y)=\varphi(x,-y)$ (in the first integral, a localization to $\overline{Q_i}$ is implicit
since $\Sigma^i_\ell\subset \overline{Q_i}$ by assumption). In particular, testing with $\varphi\equiv 1$, we
find
\begin{equation*}
%\label{lungpar}
\lim_{\ell\to\infty} \frac {\haus(\overline{\Omega}\cap\Sigma^i_\ell)}{\haus(\Sigma^i_\ell)}=
\frac{|\Omega\cap Q_i|}{|Q_i|}
\end{equation*}
i.e., using the second property in (i) above,
\begin{equation}
\label{lungpar2}
\lim_{\ell\to\infty} \frac {\haus(\overline{\Omega}\cap\Sigma^i_\ell)}{\ell}=
\frac{|\Omega\cap Q_i|}{|Q_i|}
\end{equation}
(when $Q_i\subset\Omega$ this is obvious since, in this case, it is contained in (i) above).

In \eqref{mufit} we may assume that $\alpha_i>0$ for all $i$, since if $\alpha_i=0$ for some
$i$ then the right-hand side of \eqref{gammasup} is $+\infty$. Then, for large $L$
we can define the sets
\begin{equation}
\label{defSL}
\Sigma_L:=\left(\partial\Omega\right)\cup \bigl(\bigcup_{Q_i\in \chb} \Sigma^i_{s^2\alpha_i L}\cap \Omega\bigr),
\end{equation}
where $\Sigma^i_{s^2\alpha_iL}$ denotes the set $\Sigma^i_{\ell}$ defined above, when $\ell=s^2\alpha_i L$
(recall that $s^2=|Q_i|$, the area of $Q_i$).

Some remarks are in order. First, since each set $\Sigma^i_{s^2\alpha_iL}$
is connected and covers $\partial Q_i$,
building on (C1) above one can check that $\Sigma_L$ is connected
(to this purpose, the fact that $\Sigma_L\supset\partial\Omega$
is essential). Then, setting $\ell=s^2\alpha_i L$ in \eqref{lungpar2} we obtain
\begin{equation}
\label{lungpar3}
\lim_{L\to\infty} \frac {\haus(\overline{\Omega}\cap\Sigma^i_{s^2\alpha_i L})}{L}=
\frac{s^2\alpha_i |\Omega\cap Q_i|}{|Q_i|}
=
\alpha_i |\Omega\cap Q_i|
\quad\forall Q_i\in\chb.
\end{equation}
Now, observing that $\haus(\partial\Omega)$ is finite, and that the sets
$\Sigma^i_{s^2\alpha_i L}$ are pairwise disjoint, except for their overlapping along
$\bigcup \partial Q_i$ whose total length is independent of $L$,
from \eqref{defSL} and \eqref{lungpar3} we obtain
\begin{equation}
\label{lungtot}
\lim_{L\to\infty} \frac {\haus(\Sigma_L)}{L}=
\sum_{Q_i\in\chb}
\lim_{L\to\infty}
\frac {\haus(\overline{\Omega}\cap\Sigma^i_{s^2\alpha_i L})}{L}
=
\sum_{Q_i\in\chb}
\alpha_i |\Omega\cap Q_i|=1
\end{equation}
according to \eqref{constants}. Along the same lines, given $\varphi\in C(\overline{\Omega}\times
{\mathbb S}^1)$ with $\varphi(x,y)=\varphi(x,-y)$,
denoting by $\xi_L$ the unit normal to $\Sigma_L$, we may split
\begin{equation}
\begin{split}
\label{varifoldsconverge6}
&\lim_{L \to\infty} \frac 1{\haus(\Sigma_L)}
\int_{\Sigma_L} \varphi(x,\xi_L(x))\,d\haus(x) \\=
\sum_{Q_i\in \chb}
&\lim_{L \to\infty}
\frac {\haus(\Sigma^i_{s^2\alpha_i L})}{\haus(\Sigma_L)}\cdot
\frac 1{\haus(\Sigma^i_{s^2\alpha_i L})}
\int_{\overline{\Omega}\cap \Sigma^i_{s^2\alpha_i L}} \varphi(x,\xi^i_\ell(x))\,d\haus(x).
\end{split}
\end{equation}
Now for every $Q_i\in\chb$, using  \eqref{lungtot} and condition (i) above with $\ell=s^2\alpha_i L$,
we have
\[
\lim_{L \to\infty}
\frac {\haus(\Sigma^i_{s^2\alpha_i L})}{\haus(\Sigma_L)}
=s^2\alpha_i,
\]
while
putting $\ell=s^2\alpha_i L$ in \eqref{varifoldsconverge5} gives
\begin{equation*}
\lim_{L \to\infty} \frac 1{\haus(\Sigma^i_{s^2\alpha_i L})}
\int_{\overline{\Omega}\cap \Sigma^i_{s^2\alpha_i L}} \varphi(x,\xi^i_\ell(x))\,d\haus(x) =
\frac 1{|Q_i|} \int_{\Omega\cap Q_i}\int_{{\mathbb S}^1}\varphi(x,y)\,d\nu_i(y)\,dx.
\end{equation*}
Plugging the last two limits into \eqref{varifoldsconverge6}, since $s^2=|Q_i|$
we obtain
\begin{equation}\label{varifold2}
\lim_{L \to\infty} \frac 1{\haus(\Sigma_L)}
\int_{\Sigma_L} \varphi(x,\xi_L(x))\,d\haus(x) =
\sum_{Q_i\in \chb}
\alpha_i
\int_{\Omega\cap Q_i}\int_{{\mathbb S}^1}\varphi(x,y)\,d\nu_i(y)\,dx,
\end{equation}
that is \eqref{varifold}, due to the structure of $\theta$ as established in \eqref{mufit}.

Finally, to prove \eqref{gammasup}, observe that every connected component
of $\Omega\setminus \Sigma_L$ is contained inside some square $Q_j\in\chb$ and
moreover, by construction, $\Sigma_L\supset \partial(\Omega\cap Q_i)$ for \emph{every}
square $Q_i\in\chb$. As a consequence, we have
\[
\begin{split}
&\lambda_1^A(\Omega\setminus\Sigma_L)=
\min_{Q_i\in\chb}
\lambda_1^A((\Omega\cap Q_i)\setminus\Sigma_L)
=
\min_{Q_i\in\chb}
\lambda_1^A((\Omega\cap Q_i)\setminus\Sigma^i_{s^2\alpha_i L})\\
\geq
& \eta^{-1} \min_{Q_i\in\chb}
\lambda_1^{M_i}((\Omega\cap Q_i)\setminus\Sigma^i_{s^2\alpha_i L})
\geq
\eta^{-1} \min_{Q_i\in\chb}
\lambda_1^{M_i}(Q_i\setminus\Sigma^i_{s^2\alpha_i L})
\end{split}
\]
(the last two inequalities follow from \eqref{assu2} and monotonicity
of $\lambda_1$ with respect to set inclusion, respectively: observe that, since
$M_i$ is a constant matrix, the eigenvalue
$\lambda_1^{M_i}(Q_i\setminus\Sigma^i_{s^2\alpha_i L})$
makes sense for every $Q_i$, while $\lambda_1^{A}(Q_i\setminus\Sigma^i_{s^2\alpha_i L})$ would make
no sense if $Q_i\cap\partial\Omega\not=\emptyset$, since $A(x)$ is defined
only for $x\in\overline{\Omega}$). Passing to reciprocals, swapping $\limsup$ and $\max$,
and using (ii)
with $\ell=s^2\alpha_i L$, we obtain
\[
\begin{split}
&\limsup_{L\to\infty}
\frac {L^2}
{\lambda_1^A(\Omega\setminus\Sigma_L)}
\leq
\eta
\limsup_{L\to\infty}
\left(
\max_{Q_i\in\chb}
\frac {L^2}
{\lambda_1^{M_i}(Q_i\setminus\Sigma^i_{s^2\alpha_i L})}\right)\\
&\leq
\eta
\max_{Q_i\in\chb}
\left(\limsup_{L\to\infty}
\frac {L^2}
{\lambda_1^{M_i}(Q_i\setminus\Sigma^i_{s^2\alpha_i L})}\right)\\
&\leq
\frac{\eta}{s^4\alpha_i^2}
\max_{Q_i\in\chb}
\frac{s^4}{\big(\pi\int_{\mathbb{S}^1}
\sqrt{ \langle M_iy,y\rangle}\, d\nu_i(y) \big)^2}
\leq\eta F_\infty(\theta),
\end{split}
\]
having used \eqref{Finffitted} in the last passage. This proves \eqref{gammasup} up to the multiplicative
constant $\eta>1$ which, being arbitrary, can easily be remove by a diagonal argument, as in the final part
of the proof of Proposition~\ref{proptile}.
\end{proof}
The passage to general $\nu$ is standard: by the classical $\Gamma$-convergence theory the $\Gamma$-limsup inequality for every varifolds in $\prob$ holds if one proves the density in energy of those varifold considered in \eqref{mufit} in the space of all varifolds $\prob$.

\begin{lemma}[renormalization]%\label{lemmatec}
For some $a>0$, let $p:[a,\infty)\to\R^+$ be a function such that $p(L)\sim L$ as
$L\to\infty$. Then, for some $b>0$, there exists another function $q:[b,\infty)\to [a,\infty)$ such that
\begin{equation*}
%\label{tesilemma}
p(q(L))\leq L\quad\forall L\geq b,\quad\text{and}\quad
p(q(L))\sim L\quad\text{as $L\to\infty$.}
\end{equation*}
\end{lemma}
\begin{proof}
For large $L$ define the set $E(L):=\{x\geq a\,|\,\, p(x)\leq L\}$, and let $s(L):=\sup E(L)$.
%(observe that $E(L)\not=\emptyset$ if is large enough).
Since $p(L)\sim L$ as $L\to\infty$, one can easily check that also $s(L)\sim L$ as $L\to\infty$. Then, any
function $q(L)$ satisfying
\[
q(L)\in E(L)\quad
\text{and}\quad
s(L)-1 < q(L)<s(L)\qquad\text{(for every $L$ large enough)}
\]
will satisfy the claim of the Lemma, since clearly $q(L)\sim s(L)\sim L$ as $L\to\infty$, while
$q(L)\in E(L)$   guarantees  that  $p(q(L))\leq L$.
\end{proof}

\begin{remark}%\label{remrenorm}
\rm
Using the previous lemma, one can strengthen the claim of Proposition~\ref{proptile2} with
the \emph{additional} requirement that
\begin{equation}
\label{extrareq}\haus(\Sigma_L)\leq L,\quad\text{that is,}\quad \Sigma_L\in \Alo.
\end{equation}
Indeed, it suffices to apply the lemma to the function $p(L):=\haus(\Sigma_L)$, where
the sets $\Sigma_L$ are those initially yielded by Proposition~\ref{proptile2} (note that $p(L)\sim L$
by \eqref{asL}). Then, one can define the new sets $\Sigma'_L:=\Sigma_{q(L)}$, thus gaining
that $\haus(\Sigma'_L)=p(q(L))\leq L$ while keeping $\haus(\Sigma'_L)\sim L$. Then, replacing
each $\Sigma_L$ with the corresponding $\Sigma'_L$ proves the claim.
\end{remark}

\begin{proposition}[$\Gamma$-limsup inequality]\label{prop14}
For every  varifold $\theta\in\prob$, there exists
a sequence of varifolds $\{\theta_L\}\subset\prob$ such that $\theta_L\weak\theta$ and,
moreover,
\begin{equation}\label{gammasup2}
  \limsup_{L\to \infty} F_L(\theta_L)\leq F_\infty(\theta).
\end{equation}
\end{proposition}

\begin{proof}
Given $\theta\in\prob$, if $\theta$ is fitted to $\chb$ for some $s>0$, then the claim follows from
Proposition~\ref{proptile2}, strengthened according to \eqref{extrareq}. Indeed, in this
case, is suffices to define $\theta_L$ as the varifold associated with the set $\Sigma_L$,
so that the left hand side of \eqref{gammasup2} coincides with that of \eqref{gammasup}
(the fact that $\theta_L\weak \theta$ follows from \eqref{varifold2}).

Then, by well known properties of $\Gamma$-convergence, it suffices to prove that
the class of varifolds fitted to $\chb$ (for some $s>0$) is \emph{dense in energy}.
More precisely, it suffices to prove that for every varifold $\theta\in\prob$
there exist varifolds $\theta_L\in\prob$, each fitted to $\chb$ for some $s>0$ that may
depend on $L$, such that
\begin{equation}
\label{weak1}
\theta_L\weak \theta \quad\text{in $\prob$, as $L\to\infty$}
\end{equation}
and, at the same time,
\begin{equation}
\label{densen}
\limsup_{L\to\infty} F_\infty(\theta_L)\leq F_\infty(\theta).
\end{equation}
Now fix an
arbitrary $L>0$, set e.g. $s=1/L$, and consider the
 family of squares $\chb$ as in Definition~\ref{defstep}.
The main idea is to define $\theta_L$, fitted to $\chb$,
by averaging $\theta$ over a partition of $\overline{\Omega}$
essentially based on $\chb$. Observe that the open squares in $\chb$ cover $\overline{\Omega}$
only up to a set $C$ of measure zero, while the measure $\theta$ may
well charge the set $C\times{\mathbb S}^1$.
%
% On the other hand, we certainly have
% \[
% \overline{\Omega}=\bigcup_{Q_i\in\chb} \overline{\Omega\cap Q_i},
%\]
%but the closed sets $\overline{\Omega\cap Q_i}$ are \emph{not} pairwise
%disjoint (so that a sampling of $\theta$ based on the sets
%$\overline{\Omega\cap Q_i}$ might produce a $\theta_L$ with total mass larger than $1$).
%
We overcome this difficulty by choosing \emph{intermediate} Borel sets $\Omega_i$ such that
\[
\Omega\cap Q_i \,\subseteq\,\Omega_i \,\subseteq\,
\overline{\Omega\cap Q_i},\qquad
\overline{\Omega}=\bigcup_i \Omega_i,\qquad
\Omega_i\cap\Omega_j=\emptyset\quad\forall i\not=j
\]
(the actual choice of the $\Omega_i$s is irrelevant) and define
\begin{equation*}
%\label{defithetaL}
\theta_L:=\sum_{Q_i\in \chb} \alpha_i \chi_{\Omega\cap Q_i} \otimes \nu_i,
\end{equation*}
where the constants $\alpha_i$ and the probability measures $\nu_i$ are given by
\[
\alpha_i:= \frac{\theta(\Omega_i \times {\mathbb S}^1)}{|\Omega_i|}
\quad
\text{and}
\quad \nu_i(E):=
\frac{\theta(\Omega_i\times E)}
{\theta(\Omega_i \times {\mathbb S}^1)}
\quad
\text{(for every Borel set $E\subseteq {\mathbb S}^1$).}
\]
It is clear that $\theta_L$  is
a varifold fitted to $\chb$ (note that \eqref{constants} is satisfied,
since $\Omega_i$ is equivalent to $\Omega\cap Q_i$ up to a Lebesgue-negligible set).
This definition makes sense only if $\theta(\Omega_i \times {\mathbb S}^1)>0$
for every $i$: on the other hand, if $\theta(\Omega_i \times {\mathbb S}^1)=0$ for some $i$,
then $\alpha_i=0$ and $\nu_i$ becomes irrelevant (for definiteness, one can
define $\nu_i$ as \emph{any} probability measure). Observe that, in this case, one has
$F_\infty(\theta)=+\infty$, and \eqref{densen} becomes trivial, regardless of $\theta_L$.

The measure
$\theta_L$ so defined is a discretization of $\theta$ over $\chb$, and in concrete terms
its action on a Borel function $\varphi(x,y)\geq 0$ is given by
\begin{equation*}
%\label{intthetaL}
\begin{split}
\int_{\overline{\Omega}\times {\mathbb S}^1}
\varphi(x,y)\,d\theta_L(x,y)
%&=
%\sum_{Q_i\in\chb} \alpha_i \int_{\Omega\cap Q_i}
%\left(  \frac 1
%{\theta(\Omega_i \times {\mathbb S}^1)}
%\int_{\Omega_i\times {\mathbb S}^1}\varphi(x,y)\,d\theta(z,y)\right)\,dx\\
=\sum_{Q_i\in\chb} \frac 1{|\Omega_i|} \int_{\Omega\cap Q_i}
\left(
\int_{\Omega_i\times {\mathbb S}^1}\varphi(x,y)\,d\theta(z,y)\right)\,dx\\
\end{split}
\end{equation*}
It is routine  to check that \eqref{weak1} holds: e.g.,
the Wasserstein distance (see  \cite{amgisa})
$W^{1}(\theta,\theta_L)$ does not exceed $\sqrt{2}/L$, the diameter of each $Q_i$.
As a consequence, we only focus on \eqref{densen} (assuming all $\alpha_i>0$, otherwise
\eqref{densen} is trivial as already observed).

Now let $\theta=\mu\otimes \nu_x$ be the disintegration of $\theta$ (as discussed in Definition~\ref{defFi}),
and let $f\in L^1(\Omega)$ be the density of $\mu$ with respect to the Lebesgue measure
(clearly, $\mu\geq f$ as measures).
For every $Q_i\in\chb$, and for every $x\in \Omega\cap Q_i$, we have
\[
\begin{split}
&\pi\alpha_i \int_{\mathbb{S}^1}
\sqrt{ \langle A(x)y,y\rangle}\, d\nu_i(y)
=
\frac \pi {|\Omega_i|}
\int_{\Omega_i\times {\mathbb S}^1}
\sqrt{ \langle A(x)y,y\rangle}  \,d\theta(z,y)
\\
=&\frac \pi {|\Omega_i|}
\int_{\Omega_i} \left(\int_{{\mathbb S}^1}
\sqrt{ \langle A(x)y,y\rangle}
\,d\nu_z(y)\right)
d\mu(z)\\
\geq
&\frac \pi {|\Omega_i|}
\int_{\Omega_i} \left(\int_{{\mathbb S}^1}
\sqrt{ \langle A(x)y,y\rangle}
\,d\nu_z(y)\right)
f(z)\,dz.
\end{split}
\]
Now, for every small $\eps>0$, since the matrices $A(x)$ are uniformly continuous and positive
definite over $\overline{\Omega}$, if $L$ is large enough (and consequently
the side length $s=1/L$ of every $Q_j\in\chb$ is small enough) we have
\[
x,z\in \Omega_i\Rightarrow
\sqrt{ \langle A(x)y,y\rangle}
\geq (1-\eps)\sqrt{ \langle A(z)y,y\rangle}\quad
\forall y\in{\mathbb S}^1.
\]
This allows us to replace, up to a factor $1-\eps$, $A(x)$ with $A(z)$ in the previous estimate:
thus, recalling  \eqref{gammalimit}, we obtain
\[
\begin{split}
&\pi\alpha_i \int_{\mathbb{S}^1}
\sqrt{ \langle A(x)y,y\rangle}\, d\nu_i(y)
\geq
\frac {(1-\eps)\pi} {|\Omega_i|}
\int_{\Omega_i}  f(z)\left(\int_{{\mathbb S}^1}
\sqrt{ \langle A(z)y,y\rangle}
\,d\nu_z(y)\right)
\,dz\\
\geq &(1-\eps)\pi\,\essinf_{z\in\Omega}
\left(f(z)\int_{{\mathbb S}^1}
\sqrt{ \langle A(z)y,y\rangle}
\,d\nu_z(y)\right)=
\frac{1-\eps}{\sqrt{F_\infty(\theta)}},
\end{split}
\]
valid for $L$ large enough (depending only on $\eps$).
Squaring and passing to reciprocals, the arbitrariness of $x$ gives
\[
\begin{split}
\esssup_{x\in\Omega\cap Q_i}\frac {1}
{\pi^2\alpha_i^2 \left(\int_{\mathbb{S}^1}
\sqrt{ \langle A(x)y,y\rangle}\, d\nu_i(y)\right)^2}
\leq
\frac{F_\infty(\theta)}{(1-\eps)^2},\quad\forall Q_i\in\chb.
\end{split}
\]
Taking the maximum over $Q_i$ and using \eqref{Finffitted}, one obtains \eqref{gammasup2}
up to a multiplicative factor $(1-\eps)^{-2}$, which can then removed by the
usual diagonal argument.
\end{proof}

\section{Proofs of the main results}

We are in a position to prove all the results stated in the introduction.

\begin{proof}[Proof of Theorem \ref{aniso}] The claim follows immediately
as a consequence of Proposition~\ref{prop8} combined with Proposition~\ref{prop14}.
\end{proof}

\begin{proof}[Proof of Theorem \ref{teomin}] Consider a generic varifold $\theta\in \prob$,
disintegrated  as $\theta=\mu\otimes \nu_x$ as in
Definition~\ref{defFi}. A look at \eqref{gammalimit} reveals that
the following conditions are in any case \emph{necessary}, for
$\theta$ to be a (candidate) minimizer of $F_\infty$:
{\setlength\itemindent{0pt}}\begin{itemize}
\item[(a)]
 the density $f_\mu$ must satisfy $f_\mu(x)>0$ for a.e. $x\in\Omega$, otherwise the
 denominator in \eqref{gammalimit} vanishes on a set of positive measure, and
$F_\infty(\theta)=+\infty$;
\item[(b)]
the first marginal $\mu\in\Prob(\cOmega)$ must be \emph{absolutely continuous} w.r.to the
Lebesgue measure, i.e. $\mu=f_\mu(x)\,dx$,
in such a way that $\int_\Omega f_\mu=1$: otherwise, the varifold $\widehat{\theta}
:=\widehat{f}(x)\otimes \nu_x$, where $\widehat{f}:=f_\mu/\Vert f_\mu\Vert_{L^1}$ is the
renormalization of $f_\mu$, would be such that $F_\infty(\widehat{\theta})<F_\infty(\theta)$;
\item[(c)] for a.e. $x\in\Omega$, the measure $\nu_x\in\Prob(\sfera)$ must be supported on a set
of normalized eigenvectors of $A(x)$ relative to its \emph{largest eigenvalue} $\sigma_{\max}(x)$.
Indeed, for any probability measure $\nu\in\Prob(\sfera)$
one has the double inequality
\begin{equation}
\label{integrals2}
\sqrt{\sigma_{\min}(x)}\leq
\int_\sfera \sqrt{\langle A(x) y,y\rangle} \,d\nu(y)
\leq \sqrt{\sigma_{\max}(x)},
\end{equation}
where $\sigma_{\min}(x)$ is the smallest eigenvalue of $A(x)$. If $\sigma_{\min}(x)=
\sigma_{\max}(x)$, i.e. if $A(x)$ is a multiple of the identity matrix, then both inequalities
are in fact equalities, and in this case the condition that $\nu_x$ be supported
on a set of eigenvectors relative to $\sigma_{\max}(x)$ is \emph{trivial}, since any $y\in\sfera$
is such an eigenvector. On the other hand, if
$\sigma_{\min}(x)<
\sigma_{\max}(x)$ (i.e. if $A(x)$ is \emph{anisotropic}) then the second inequality in
\eqref{integrals2} becomes an equality only when the measure $\nu$ is supported in $\{\xi,-\xi\}$,
where $\xi\in\sfera$ is such that $A(x)\xi=\sigma_{\max}(x) \xi$ (this eigenvector is
now unique up to the orientation, and this completely determines $\nu$ as a probability
measure on the projective space).
In order to minimize the essential supremum
in \eqref{gammalimit}, it is necessary that the second inequality in \eqref{integrals2}
becomes an equality for a.e. $x\in\Omega$: this is the claimed condition
on the support of $\nu_x$.
\end{itemize}
Thus, in the light of these necessary
conditions, any candidate minimizer has the form $\theta=f_\mu(x)\,dx \otimes \nu_x$,
with $f_\mu>0$ such that $\int_\Omega f_\mu=1$ and the $\nu_x$ as in (c).
For any such varifold $\theta$, from \eqref{gammalimit} we find
\[
F_\infty(\theta)=
\frac{1}{\pi^2}\esssup_{x\in\Omega}\frac{1}
{\big(f_\mu(x) \sqrt{\sigma_{\max}(x)}  \big)^2}
=
\frac{1}{\pi^2}\esssup_{x\in\Omega}\frac{1}
{f_\mu(x)^2 \sigma_{\max}(x)},
\]
and by Cauchy-Schwarz it is immediate to check that the only optimal choice for $f_\mu(x)$,
under  the constraints $f_\mu>0$ and $\int_\Omega f_\mu=1$, is that in \eqref{defbestf}.

Combining with condition (b) above, we see that the first marginal $\mu$ of any
minimizer is necessarily given by $f_\infty(x)\,dx$. If, moreover, $A(x)$
is purely anisotropic, condition (c) forces the measures $\nu_x$ to be uniquely determined
(as measures on the projective space): as a consequence, the minimizer $\theta$ is uniquely determined
as an element of $\prob$.
\end{proof}

\begin{proof}[Proof of Theorem~\ref{teodist}] By Corollary \ref{cordist}, up to a subsequence
(not relabelled) we have $\theta_{\Sigma_L}\weak \theta_\infty$ in $\prob$, where $\theta_\infty$ is
an absolute minimizer of $F_\infty$. Recalling \eqref{weakconvergenceT}, this means that
\begin{equation}
\label{weakconvergenceT2}
\lim_{L\to\infty}
\int_{\Ots} \varphi(x,y)
 \,d\theta_{\Sigma_L}
=
\int_{\Ots} \varphi(x,y)
\,d\theta_\infty\quad
\forall
\varphi\in\Csym.
\end{equation}
Moreover, by the last part of Theorem~\ref{teomin},
the first marginal $\mu$ of $\theta_\infty$ is \emph{uniquely} determined as a probability measure over
$\cOmega$: it is the absolutely
continuous measure with density $f_\infty(x)$ defined in \eqref{defbestf}. Now fix a square $Q\subset\Omega$.
Since $\mu$ is absolutely continuous w.r.to the Lebesgue measure in $\Omega$ and $|\partial Q|=0$,
we have $\theta_\infty(\partial Q\times \sfera)=0$, that is, the measure $\theta$ does not
charge the set $(\partial Q)\times\sfera$. This set is the boundary (relative to $\cOmega\times \sfera$)
of the cylinder $Q\times\sfera$: therefore, from standard results in measure theory, the weak convergence
in \eqref{weakconvergenceT2} can be localized to the cylinder $Q\times\sfera$, that is,
we have
\begin{equation*}
%\label{weakconvergenceT2}
\lim_{L\to\infty}
\int_{Q\times\sfera} \varphi(x,y)
 \,d\theta_{\Sigma_L}
=
\int_{Q\times\sfera} \varphi(x,y)
\,d\theta_\infty\quad
\forall
\varphi\in C_{\text{sym}}(\cOmega\times \sfera).
\end{equation*}
More explicitly, according to \eqref{defT} written with $\Sigma_L$ in place of $\Sigma$,
this means that
\begin{equation*}
%\label{normopt2}
\lim_{L\to\infty}\frac 1  {\haus(\Sigma_L)}
\int_{Q\cap\Sigma_L} \varphi(x,\xi_{\Sigma_L}(x)) \,d\haus(x)
=
\int_{Q}
\left(
\int_\sfera \varphi(x,y)\,d\nu_x(y)\right)
f_\infty(x)\,dx,
\end{equation*}
having used the slicing $\theta_\infty=f_\infty(x)dx\otimes \nu_x$, provided by Theorem~\ref{teomin},
in the right hand side. In particular, \eqref{asopt} follows if we choose $\varphi\equiv 1$.

Similarly, choosing $\varphi(x,y)=\psi(y)$ where
$\psi\in C(\sfera)$ satisfies $\psi(-y)=\psi(y)$, we have
\begin{equation*}
\lim_{L\to\infty}\frac 1  {\haus(\Sigma_L)}
\int_{Q\cap\Sigma_L} \psi(\xi_{\Sigma_L}(x)) \,d\haus(x)
=
\int_{Q}
\left(
\int_\sfera \psi(y)\,d\nu_x(y)\right)
f_\infty(x)\,dx.
\end{equation*}
If $Q$ is contained in the anisotropy region where $\sigma_{\min}(x)<\sigma_{\max}(x)$,
then by Theorem~\ref{teomin} for a.e. $x\in Q$ the probability measure $\nu_x$ is supported on
the set $\{ \xi(x)$, $-\xi(x)\}$, where $\xi(x)$ is the eigenvector of $A(x)$ relative
to its largest eigenvalue $\sigma_{\max}(x)$. Therefore, since $\psi(y)=\psi(-y)$, the
last equation simplifies to
\begin{equation*}
\lim_{L\to\infty}\frac 1  {\haus(\Sigma_L)}
\int_{Q\cap\Sigma_L} \psi(\xi_{\Sigma_L}(x)) \,d\haus(x)
=
\int_{Q}
\psi\bigl(\xi(x)\bigr)
f_\infty(x)\,dx.
\end{equation*}
If we multiply and divide the left hand side by $\haus(Q\cap\Sigma_L)$, then
\eqref{normopt} follows, using \eqref{asopt}.

Finally, using the explicit structure of the minimizers $\theta_\infty$
provided by Theorem~\ref{teomin},
from \eqref{defbestf} we find
\[
\min_{\theta\in\prob} F_\infty(\theta)=F_\infty(\theta_\infty)=
\frac{\left(\int_{\Om}1/\sqrt{\smax(x)}dx\right)^2}{\pi^2}.
\]
Then
\eqref{asymptotically} follows immediately, from basic $\Gamma$-convergence theory.
\end{proof}

\bigskip
\textbf{Acknowledgements.}
The second author wishes to express his gratitude to Gianni Dal Maso for some discussions on the problem studied in this paper.

\end{document}